\documentclass[12pt]{amsart}

\usepackage{amsmath,amsthm,amsopn,amssymb,mathrsfs,times,newtxtext,newtxmath,upref}
\usepackage[mathscr]{eucal}

\usepackage{enumitem}
\setlist{label={$$\roman{enumi}\kern1pt$)$}}

\newtheorem{thm}{Theorem}[section]
\newtheorem{prop}[thm]{Proposition}
\newtheorem{cor}[thm]{Corollary}
\newtheorem*{cor*}{Corollary}
\newtheorem{lema}[thm]{Lemma}
\newtheorem*{lema*}{Lemma}

\numberwithin{equation}{section}
\theoremstyle{definition}
\newtheorem*{Def}{Definition}

\newtheorem{Example}{Example}
\newtheorem*{obs}{Remark}
\newtheorem*{nota}{Notation}

\numberwithin{equation}{section}

\newcommand{\PI}[2]{\left\langle \,#1 , #2\, \right\rangle}
\newcommand{\K}[2]{\left[\,#1 , #2\, \right]}
\newcommand{\pas}{P_{W,\St}}
\newcommand{\cpas}{P_{W,\ol{\St}}}
\newcommand{\St}{\mathcal{S}}
\newcommand{\V}{\mathcal{V}}

\newcommand{\E}{\mathcal{E}}
\newcommand{\HH}{\mathcal{H}}
\newcommand{\KK}{\mathcal{K}}
\newcommand{\M}{\mathcal{M}}

\newcommand{\N}{\mathcal{N}}

\newcommand{\W}{\mathcal{W}}

\newcommand{\mc}[1]{\mathcal{#1}}
\newcommand{\T}{\mathcal{T}}
\newcommand{\ol}{\overline}
\newcommand{\clran}{\ol{\mathrm{ran}}\,}

\newcommand{\cldom}{\ol{\mathrm{dom}}\,}
\newcommand{\clmul}{\ol{\mathrm{mul}}\,}

\newcommand{\ra}{\rightarrow}
\DeclareMathOperator{\Sp}{Mp}
\DeclareMathOperator{\SP}{MP}

\DeclareMathOperator{\MP}{MP}

\newcommand{\PMN}{P_{\M,\N}}

\newcommand{\spl}{sp}
\newcommand{\WS}{\St(W)}
\newcommand{\WSb}{\St^\bot(W)}

\newcommand{\matriz}[4]{\displaystyle\
	\left(
	\begin{array}{cc}
		{#1}&{#2}\\
		{#3}&{#4}
	\end{array}
	\right)}

\newcommand{\vecdm}[2]{\displaystyle\
	\begin{array}{cc}
		{#1}\\{#2}
\end{array}}

\DeclareMathOperator{\ran}{ran}
\DeclareMathOperator{\dom}{dom}

\DeclareMathOperator{\mul}{mul}

\newcommand{\pl}{{\mathbin{/\mkern-3mu/}}}

\newcommand{\lr}{\mathrm{lr}}

\begin{document}


\baselineskip=17pt


\title[Matrix representations of multivalued projections and LS problems]{Matrix representations of multivalued projections and least squares problems}

\author[M. L. Arias]{M. Laura Arias}
\address{Instituto Argentino de Matemática Alberto P. Calderón- CONICET\\
Buenos Aires, Argentina \\ and\\ Depto. de Matemática, Facultad de Ingeniería, Universidad de Buenos Aires\\
Buenos Aires, Argentina}
\email{lauraarias@conicet.gov.ar}

\author[M. Contino]{Maximiliano Contino}
\address{Depto. de An\'alisis Matem\'atico, Facultad de Matem\'aticas, Universidad Complutense de Madrid, Madrid, Spain \\ and \\ Instituto Argentino de Matemática Alberto P. Calderón- CONICET\\
	Buenos Aires, Argentina \\ and \\ Depto. de Matemática, Facultad de Ingeniería, Universidad de Buenos Aires\\
	Buenos Aires, Argentina}
\email{mcontino@fi.uba.ar }

\author[A. Maestripieri]{Alejandra Maestripieri}
\address{Instituto Argentino de Matemática Alberto P. Calderón- CONICET\\
	Buenos Aires, Argentina \\ and \\ Depto. de Matemática, Facultad de Ingeniería, Universidad de Buenos Aires\\
	Buenos Aires, Argentina}
\email{amaestri@fi.uba.ar}

\author[S. Marcantognini]{Stefania Marcantognini}
\address{Instituto Argentino de Matemática Alberto P. Calderón- CONICET\\
	Buenos Aires, Argentina\\ and \\
	Universidad Nacional de General Sarmiento – Instituto de Ciencias\\
	Los Polvorines, Pcia. de Buenos Aires, Argentina}
\email{smarcantognini@ungs.edu.ar}

\begin{abstract} Multivalued projections are applied to the study of weighted least squares solutions of linear relations equations (or inclusions) and some of its applications. 
	To this end a matrix representation of multivalued projections with respect to the closure of their ranges is described. 
\end{abstract}

\maketitle


\renewcommand{\thefootnote}{}



\renewcommand{\thefootnote}{\arabic{footnote}}
\setcounter{footnote}{0}


\section{Introduction}
In many approximation problems, associated spaces often decompose as the sum of two subspaces which are not complementary (in the sense that their intersection is nontrivial). 
Examples of this include least squares problems with a positive weight \cite{Elden, CM} or selfadjoint weight  \cite{Hassibi1, Hassibi2, patel2002solving, giribet2010indefinite}, where linear estimation in Krein spaces with applications to $\HH^{\infty}$-estimation and control theory are considered. Positive definite and indefinite abstract splines also exhibit this behaviour \cite{Atteia, CMS, delvos1978splines} and \cite{Canu1,  giribet2010indefinite2}, as do their regularized versions (known as smoothing problems) \cite{Tikhonov}. Other examples occur in sampling and reconstruction  \cite{dvorkind2009robust}, as well as in machine learning \cite{Canu2}, to list just a few.

Most of these problems were studied using matrix and operator techniques.
However, multivalued projections provide for instance a more natural way of describing weighted least squares solutions of equations. 
A \emph{multivalued projection}, or  \emph{semi-projection}, $E$ acting on $\HH$ is a subspace of $\HH \times \HH$ (commonly referred to as a \emph{linear relation}) which is \emph{idempotent} ($E^2=E$) and with the property that $\ran E$ (the range of $E$) is included in $\dom E$ (the domain of $E$). Any multivalued projection $E$ is completely determined by $\ran E$ and  $\ker E$ (the kernel of $E$); its multivalued part and domain being $\ran E \cap \ker E,$  $\ran E + \ker E,$ respectively. A detailed account of multivalued projections can be found in \cite{Cross1}, \cite{Cross} and \cite{Labrousse}.

When looking for least squares solutions of the equation $Ax=b,$ for a given vector $b\in \HH,$ one can consider $A$ as a relation in $\HH \times \HH$ instead of an operator. In this case, $Ax$ is a linear manifold and the equation ``$Ax=b$'' becomes the relation equation or \emph{the inclusion} ``$b \in Ax$''.	The study of least squares solutions of inclusions with a weight then provides a general framework for all of the previous problems. 

Least squares solutions of inclusions were studied by Z. Nashed \cite{LN} and T. \'Alvarez \cite{TA} and these sorts of approximation problems appear in control theory \cite{Lee, lee1989constrained}. Here, we consider least squares solutions of inclusions with a positive weight. More precisely, given a linear relation $A$ in a Hilbert space $\HH,$ a vector $b \in \HH,$ and a positive bounded linear operator  $W$  acting on $\HH,$ the vector $x_0$ is a weighted least squares solution of the inclusion $b \in Ax$ if $x_0 \in \dom A$ and there exists $z \in Ax_0$ such that
\begin{equation}\label{ALSSintro}
	||z-b||_{W}= \underset{y \in \ran A}{\min} \ \ ||y-b||_{W},
\end{equation}
where $\Vert \cdot \Vert_W$ is the semi-norm arising from the (semi-definite) inner product $\PI{W\cdot}{\cdot}.$  To analize \eqref{ALSSintro} it would be useful to have a $W$-orthogonal projection onto $\ran A.$ But,  with respect to this inner product, the orthogonal companion of a subspace $\St$ need not be a complement, since the sum may not be the whole space, and even when it is, it may not be direct. 
All this information is encoded in a multivalued projection $\pas$ that depends on $\St$ and the weight $W.$

In dealing with \eqref{ALSSintro}, it is helpful to have a certain matrix representation for multivalued projections; namely, a matrix representation with respect to the closure of their ranges. A bounded projection operator $E$ always admits the matrix representation with respect to the decomposition $\HH=\ran E \oplus\ran E^{\perp}$
\begin{equation}\label{mstandardintro}
	E=\matriz{I}{x}{0}{0},
\end{equation}
where  $x: \ran E^{\perp} \ra \ran E$ is a bounded operator. 
Multivalued projections may not have such matrix representations and even when they do, they may not be unique. However, under certain conditions, a mutivalued projection does admit a representation with respect to the decomposition $\HH=\clran E\oplus (\ran E)^{\perp}$ that generalizes \eqref{mstandardintro}. In particular, if $\ran E$ is closed this representation always exists and  it reduces to \eqref{mstandardintro} where $x$ is a linear relation in $\ran E^{\perp} \times \ran E.$ 

In what follows, notation and background material on linear relations are given in Section \ref{two}. Section \ref{three} starts by looking at those multivalued projections which can be represented as $2\times 2$ matrices with respect to the closures of their ranges, and as before, the $(1,2)-$coefficients play a determining role. Thus we get a similar representation to (\ref{mstandardintro}) as well as a version in our context of Ando’s formula for the $(1,2)-$coefficient \cite{Ando}.  These  results are extended to the linear relations $E$ such that $E\subseteq E^2$, the so-called \emph{super-idempotents} \cite{idem}. In turn the idempotents admitting such a representation are also characterized.

Section \ref{fourth} is devoted to describing the multivalued projection $\pas,$ where $W$ is a selfadjoint bounded operator acting on $\HH,$ $\St$ is a subspace of $\HH$ and $\pas$ is the multivalued projection with range $\St$ and kernel the  $W$-orthogonal companion of $\St.$ 
A notion that is useful in this context is that of \emph{complementability} of $W$ with respect to $\St,$ defined by Ando for matrices \cite{AndoC} and extended to operators  in  \cite{Szeged} and \cite{bilateral}. Complementability is equivalent to $\pas$ having domain the whole space. Basic properties of $\pas$ include that it is $W$-symmetric and is $W-$selfadjoint whenever $W$ is $\St-$complementable.  When $\St$ is closed, $\pas$ has a representation as in \eqref{mstandardintro} and the $(1,2)-$coefficient gives a criterion for the $\mathcal S-$complementability of $W;$ explicitly, $W= \left(\begin{smallmatrix} a & b \\ b^* & c \end{smallmatrix}\right)$ is $\mathcal S-$complementable if and only if $ax=b$
or equivalently $x=a^{-1}b$ where $a^{-1}$ is the inverse of $a$ as a relation. Also, when $W$ is positive semi-definite and $\mathcal S-$complement\-able, the formula $W(I - P_{W,\mathcal S})$ for the Schur complement of $W$ to $\St^{\perp}$ is obtained.

In Section \ref{fifth} we turn our attention to the weighted least squares solutions of the inclusion problem \eqref{ALSSintro}. The study is carried out using the multivalued projection $P_{W, \mathcal S}$ with $W$ the weight and $\mathcal S := \mbox{\rm ran }A$. There is then a solution to \eqref{ALSSintro} if and only if $b \in \dom \pas.$  Since the multivalued part of $\pas$ is $\St \cap \ker W,$ the relation $W^{1/2}(I-\pas)$ turns out to be an operator and, when $b \in \dom \pas$, the minimum in \eqref{ALSSintro} is given by $\Vert W^{1/2}(I-\pas)b \Vert.$ In particular, there is a solution to (\ref{ALSSintro}) for all $b\in \HH$ if and only if $W$ is $\St-$complementable. We prove that many of the classical results of least squares solutions of equations have an analogue for inclusions. For example, if $b \in \mbox{\rm dom }P_{W, \mathcal S}$, the set of weighted least squares solutions of the inclusion $b \in Ax$ is $x_0 + A^{-1}\mbox{\rm ker } W$ with $x_0$ any particular solution. Also, least squares solutions satisfy a normal equation similar to the classical one, in terms of relations. Finally the analysis of the weighted least squares inclusion problem is applied to the abstract spline problem and the associated Tikhonov regularization or smoothing problem.

\section{Preliminaries} \label{two}
\label{sec:preliminaries}

Throughout, $\HH,$ $\KK$ and  $\mc E$  are complex and separable Hilbert spaces.  The orthogonal   sum of two subspaces $\M$ and $\N$ of a Hilbert space $\HH$ is $\M \oplus \N.$ 
The orthogonal complement of a subspace $\mc{M} \subseteq \HH$ is written as $\mc{M}^\perp,$ or $\HH \ominus \M$ interchangeably. If $\M$ is closed,  $P_\M$ denotes the
orthogonal projection onto $\M.$

We frequently use the following result \cite[Theorem 13]{Deutsch}.
\begin{prop}\label{Deutch} Let $\M, \N$ be closed subspaces of $\HH$ then $\M+\N$ is closed if and only if $\M^{\perp}+\N^{\perp}$ is closed.
\end{prop}

\begin{lema} \label{lemasubs} Let $\M$ and $\St$ be subspaces of $\HH$ with $\St$ closed.
	Then the following are equivalent:
	\begin{enumerate}
		\item $P_{\St} (\M) \subseteq \M;$
		\item  $\M=\St \cap \M \oplus  \St^{\perp}\cap \M;$
		\item $P_{\St}(\M)=\St \cap \M.$ 
	\end{enumerate}
\end{lema}

\begin{proof} Straightforward.
\end{proof}

We  consider the standard inner product on $\HH \times \KK$ 
$$\PI{(h,k)}{(h',k')}=\PI{h}{h'}+\PI{k}{k'}, \ (h,k), (h',k') \in \HH \times \KK,$$
with the associated norm  $\Vert (h,k)\Vert=(\Vert h \Vert^2+\Vert k \Vert^2)^{1/2}.$

A linear relation from  $\HH$ into $\KK$ is a linear subspace of the Cartesian product $\HH \times \KK.$ The set of linear relations from $\HH$ into $\KK$ is denoted by $\lr(\HH,\KK),$ and $\lr(\HH)$ when $\HH=\KK.$ 
The domain,  range, kernel or nullspace and multivalued part of $T\in \lr(\HH,\KK)$ are denoted by $\dom T,$  $\ran T,$ $\ker T$ and  $\mul T:=\{y \in \KK: (0,y) \in T\},$ respectively. When $\mul T=\{0\},$ $T$ is (the graph of) an operator. 

The space of bounded linear operators from $\HH$ to $\KK$ is written as $L(\HH, \KK),$ or $L(\HH)$ when $\HH=\KK.$

The next lemma was stated by Arens \cite[2.02]{Arens}. See also \cite[Proposition 1.21]{Labrousse}.
\begin{lema}\label{lemalr} Let $S, T \in \lr(\HH,\KK).$ Then $S=T$ if and only if $S\subseteq T,$ $\dom T\subseteq\dom S$ and $\mul T \subseteq \mul S.$
\end{lema}

Consider $T, S \in \lr(\HH,\KK).$ Denote by $T \cap \St$  and $T \ \hat{+}  \ S$ the usual intersection and sum of $T$ and $S$ as subspaces, respectively.
In particular, $\mul(T  \cap  S)=\mul T\cap \mul S$ and $\ker(T  \cap  S)=\ker T \cap \ker S,$  $\dom(T \ \hat{+}  \ S)=\dom T+\dom S$ and $\ran(T \ \hat{+}  \ S)=\ran T+\ran S.$ 
We write  $T \hat{\oplus} S$ if $T \subseteq S^{\perp}.$ 

The sum of $T$ and $S$ is the linear relation defined by
$$T+S:=\{(x,y+z): (x,y ) \in T \mbox{ and } (x,z) \in S\}.$$
In particular, $\dom(T+S)=\dom T \cap \dom S$ and $\mul(T+S)=\mul T+ \mul S.$

If $R \in \lr(\KK,\mc E)$  the product $RT$ is the linear relation from $\HH$ to $\mc E$ defined by
$$RT:=\{(x,y): (x,z) \in T \mbox{ and } (z,y) \in R \mbox{ for some } z \in \mc{K}\}.$$ Then $-T=-IT=\{(u,-v): (u,v) \in T\}.$

The inverse of $T$ is the linear relation defined by $T^{-1}:=\{(y,x) : (x,y) \in T\}$ so that $\dom T^{-1}=\ran T$ and $\mul T^{-1} =\ker T.$

Given a subspace $\M$ of $\HH$, $I_\M:=\{(u,u): u\in\M\}$ and $0_\M:=\M \times \{0\}.$ When $\M=\HH$ we write $I$ and $0$ instead. 

Set
$$
T(\mc{M}):=\{y : (x,y) \in T \mbox{ for some } x \in \mc{M} \},
$$  in particular, for  $x \in \HH,$ $Tx:=T(\{x\})=\{y\in \KK: (x,y)\in T\}$. It holds that  $Tx=y+\mul T$ for any $(x,y) \in T.$ 

Write
$$T|_{\M}:=\{(x,y) \in T : x \in \M\}=T\cap (\M\times \HH).$$
Then
\begin{equation} \label{Productres}
	(ST)|_{\M}=S(T|_{\M}).
\end{equation}

\begin{lema} \label{propbasicas} Let $T \in \lr(\HH,\mc K),$ $R\in \lr(\mc K,\mc E).$ Then
	\begin{enumerate}
		\item[1.] $\ran(RT)=R(\ran T).$
		\item[2.] $\mul(RT)=R(\mul T).$
	\end{enumerate} 
\end{lema}

The \emph{closure} $\ol{T}$ of $T$ is the closure of the  subspace $T$ in $\HH \times \KK.$ 
The relation $T$ is  \emph{closed} when it is closed as a subspace of $\HH \times \KK.$ 
The \emph{adjoint} of $T\in \lr(\HH,\KK)$ is the linear relation from $\KK$ to $\HH$ defined by
$$T^*:=\{(x,y) \in \KK \times \HH: \PI{g}{x}=\PI{f}{y} \mbox{ for all } (f,g) \in T \}.$$ It holds that
$$T^*=JT^{\perp}=(JT)^{\perp},$$ where $J(x,y)=i(-y,x).$ 
The adjoint of $T$
is  a closed linear relation,  $\ol{T}^*=T^*$ and $T^{**}:=(T^*)^*=\ol{T}.$ It holds that $\mul T^* =(\dom T)^{\perp}$ and $\ker T^* =(\ran T)^{\perp}.$ Therefore, if $T$ is closed both $\ker T$ and $\mul T$ are closed subspaces.

If $T \in \lr(\HH,\mc K)$ and $R\in \lr(\mc K,\mc E)$  then
\begin{equation} \label{product}
	T^*R^* \subseteq (RT)^*	
\end{equation}

\begin{lema}[{\cite[Lemma 2.9]{Derkach}}] \label{productolema} If $T \in \lr(\HH,\mc K)$ and $R \in \lr(\mc{K}, \mc{E})$ is closed with closed domain and $\ran T \subseteq \dom R$  then there is equality in \eqref{product}.
	In particular, if $R \in L(\mc{K}, \mc{E})$ there is equality in \eqref{product}.
\end{lema}

\subsection{Matrix representations of linear relations}

Here we collect several results from \cite{HassiMR} regarding the operational calculus for block matrices with entries that are linear relations. 

Let $\St$ be a closed subspace of $\HH$ and let $a \in \lr(\St),$ $b\in \lr(\St^{\perp},\St),$ $c \in \lr(\St,\St^{\perp})$ and $d \in \lr(\St^{\perp})$ be linear relations. 

The \emph{column} of $a$ and $c$ is the linear relation in $\St \times \HH$  defined by
$$\begin{pmatrix}
	a \\ c
\end{pmatrix}_\St:=\left\{ \left(x, w_1+w_2 \right): (x,w_1) \in a, (x,w_2) \in c\right\}$$ and the \emph{row} of $a$ and $b$ is the linear relation in $\HH \times \St$ defined by
$$\begin{pmatrix}
	a & b
\end{pmatrix}_\St:=\left\{ \left(x_1+x_2, w+z \right): (x_1,w) \in a, (x_2,z) \in b\right\}.$$

The  linear relation in $\HH \times \HH$ \emph{generated} by the blocks $a,$  $b,$ $c$ and $d$ is defined as
$$\matriz{a}{b}{c}{d}_{\St}\!\!:=\left\{ \left(x_1+x_2, (w_1+z_1)+(w_2+z_2)\right): \! \!\!\!\vecdm{(x_1,w_1) \in a, (x_2,z_1)\in b}{(x_1,w_2) \in c, (x_2,z_2)\in d}\right\}.$$

If $T:=\matriz{a}{b}{c}{d}_{\St}$ then $\dom T=\dom a \cap \dom c \oplus \dom b \cap \dom d$ and $\mul T=(\mul a +\mul b) \oplus (\mul c + \mul d).$

When no confusion arises, we shall omit the sub-index $\St$.

\begin{obs} Different block matrices may generate the same linear relation. From now on, when we write that two block matrices are equal, we mean that they generate the same linear relation.
\end{obs}

On the other hand, given a linear relation $T \in \lr(\HH)$ and  a closed subspace $\St$ of $\HH,$ we say that $T$ \emph{admits a $2 \times 2$ block matrix representation with respect to $\St$} if there exist blocks  $a \in \lr(\St),$ $b\in \lr(\St^{\perp},\St),$ $c \in \lr(\St,\St^{\perp})$ and $d \in \lr(\St^{\perp})$ such that 
$$T=\matriz{a}{b}{c}{d}_{\St}.$$
It is easy to check that
$\dom a  \cap \dom c = \St \cap \dom T$ and $\dom b  \cap \dom d = \St^{\perp} \cap \dom T,$ and $\mul a  + \mul b = \St \cap \mul T$ and $\mul c   + \mul d  = \St^{\perp} \cap \mul T.$

\begin{thm}[\normalfont{cf. \cite[Theorem 5.1]{HassiMR}}] \label{theoMLR}  Let $T \in \lr(\HH)$ and let $\St$ be a closed subspace of $\HH.$  Then the following are equivalent:
	\begin{enumerate}
		\item $T$ admits a $2 \times 2$ block matrix representation with respect to $\St;$
		\item  $P_{\St} (\dom T) \subseteq \dom T$ and $P_{\St} (\mul T) \subseteq \mul T;$ 
		\item $T$ admits the representation 
		\begin{equation} \label{Acanon}
			T=\matriz{a}{b}{c}{d}_{\St},
		\end{equation} where $a:=P_{\St}T|_{\St},$  $b:=P_{\St}T|_{\St^{\perp}},$ $ c:=P_{\St^{\perp}}T|_{\St}$ and  $d:=P_{\St^{\perp}}T|_{\St^{\perp}}.$
	\end{enumerate}
\end{thm}

\begin{lema}[\normalfont{\cite[Lemma 5.5]{HassiMR}}] \label{suma} Let $\St$ be a closed subspace of $\HH$ $a, a' \in \lr(\St),$ $b, b'\in \lr(\St^{\perp},\St),$ $c, c' \in \lr(\St,\St^{\perp}),$ $d, d' \in \lr(\St^{\perp}),$  $T=\matriz{a}{b}{c}{d}_{\St},$ $S=\matriz{a'}{b'}{c'}{d'}_{\St}$ and $\lambda \in \mathbb{C}.$ Then
	$$T+\lambda S=\matriz{a+\lambda a'}{b+\lambda b'}{c+\lambda c'}{d+\lambda d'}_{\St}.$$
\end{lema}

\begin{lema}[\normalfont{\cite[Lemma 7.2]{HassiMR}}] \label{filacolumna} Let $\St$ be a closed subspace of $\HH,$ $a \in \lr(\St),$ $c \in \lr(\St,\St^{\perp}),$ $T \in \lr(\HH,\St)$ and $C:=\begin{pmatrix}
		a \\ c
	\end{pmatrix} \in \lr(\St, \HH).$
	Then
	\begin{equation} \label{ig}
		C T \subseteq \begin{pmatrix}
			a T \\ c T
		\end{pmatrix}.
	\end{equation}
	Equality in \eqref{ig} holds if and only if
	$$\mul T \cap \dom C \subseteq  \mul T \cap \ker a + \mul T \cap \ker c.$$
\end{lema}

\begin{lema}[\normalfont{\cite[Lemma 7.1 and Corollary 7.1]{HassiMR}}] \label{filacolumna2} Let $\St$ be a closed subspace of $\HH,$ $a \in \lr(\St),$ $b \in \lr(\St^{\perp},\St),$ $T \in \lr(\St,\HH)$ and $R:=\begin{pmatrix}
		a & b
	\end{pmatrix} \in \lr(\HH, \St).$
	Then
	\begin{equation} \label{ig2}
		\begin{pmatrix}
			Ta & Tb
		\end{pmatrix} \subseteq TR 
	\end{equation}
	If $T \in L(\St,\HH)$ equality in \eqref{ig2} holds.
\end{lema}

\begin{lema}\label{producto} Let $\St$ be a closed subspace of $\HH,$ $a, a' \in \lr(\St),$ $b, b'\in \lr(\St^{\perp},\St),$ $c, c' \in \lr(\St,\St^{\perp}),$ $d, d' \in \lr(\St^{\perp})$  $T=\matriz{a}{b}{c}{d}_{\St}$ and $S=\matriz{a'}{b'}{c'}{d'}_{\St}.$ Then
	\begin{align*}
		TS=\begin{pmatrix}
			a\\c
		\end{pmatrix} \begin{pmatrix}
			a'&b'
		\end{pmatrix}+\begin{pmatrix}
			b\\d
		\end{pmatrix} \begin{pmatrix}
			c'&d'
		\end{pmatrix} \supseteq 
		\begin{pmatrix}
			\begin{pmatrix}
				a \\ c\end{pmatrix} a' 
			& \begin{pmatrix}
				a \\ c
			\end{pmatrix} b' 
		\end{pmatrix}+\begin{pmatrix}
			\begin{pmatrix}
				b \\ d\end{pmatrix} c' 
			& \begin{pmatrix}
				b \\ d
			\end{pmatrix} d' 
		\end{pmatrix}.
	\end{align*}
\end{lema}
\begin{proof} It follows by \cite[Lemma 8.2 and Equation (8.5)]{HassiMR}. 
\end{proof}

\begin{lema}[\normalfont{\cite[Lemma 4.1]{HassiMR}}] \label{rcclosed} Let $\St$ be a closed subspace of $\HH,$ $a \in \lr(\St),$ $b\in \lr(\St^{\perp},\St)$ and $c \in \lr(\St,\St^{\perp}).$ Then
	\begin{enumerate}
		\item [1.] If $a, c$ are closed then the column $\begin{pmatrix}
			a\\c
		\end{pmatrix}$ is closed.
		\item [2.]  If $a, b$ are closed, $\dom b^* \subseteq \dom a^*$ and $\dom a^*$ is closed, then the row $\begin{pmatrix}
			a&b
		\end{pmatrix}$ is closed.
	\end{enumerate}
\end{lema}

\begin{prop} \label{starMR} Let $\St$ be a closed subspace of $\HH,$ $a \in \lr(\St),$ $b\in \lr(\St^{\perp},\St),$ $c \in \lr(\St,\St^{\perp}),$ $d \in \lr(\St^{\perp})$ and 
	$$T=\matriz{a}{b}{c}{d}_{\St}.$$
	If $\dom a\subseteq \dom c,$ $\clmul c=\mul \ol{c},$ $\dom \ol{c}$ is closed , $\dom b\subseteq \dom d,$ $\clmul d=\mul \ol{d}$ and $\dom \ol{d}$ is closed then
	$$T^*=\matriz{a^*}{c^*}{b^*}{d^*}_{\St}.$$ 
\end{prop}
\begin{proof}
	It follows by \cite[Lemma 4.1 and Corollary 6.1]{HassiMR}.
\end{proof}

\subsection{Multivalued projections}

\begin{Def} Given $E \in \lr(\HH),$ we say that $E$ is a \emph{multivalued projection} if $E^2=E$ and $\ran E \subseteq \dom E;$ if $E$ is a multivalued projection with $\mul E=\{0\},$ then $E$ is a \emph{projection}. 
\end{Def}

Multivalued projections (or \emph{semi-projections}) were studied in \cite{Cross} and \cite{Labrousse}. Therein,  it was proved that a multivalued projection is determined by its range and kernel. More precisely, if $\Sp(\HH)$ denotes the set of multivalued projections, then

\begin{prop}[\cite{Cross,Labrousse}] \label{L1} $E \in \Sp(\HH)$ if and only if 
	$$E=I_{\ran E} \ \hat{+} \ (\ker E \times \{0\}).$$
\end{prop}
It follows from the above formula that $\dom E= \ran E + \ker E$ and $\mul E=\ran E \cap \ker E.$

Given $\M, \N$ subspaces of $\HH$  write  $$P_{\M, \N}:=I_\M \ \hat{+} \ (\N \times \{0\}).$$ Then
$$\Sp(\HH)=\{P_{\M,\N} : \M, \N \mbox{ are subspaces of } \HH\}.$$

If $\PMN$ is a projection, we write $P_{\M \pl \N},$ and if $\N=\M^{\perp},$ we write $P_\M.$

The closure and the adjoint of a multivalued projection are again multivalued projections, as the following formulae show \cite{Cross, Labrousse}. Denote by $\SP(\HH)$ the set of closed multivalued projections.

\begin{prop} \label{L2} Given $\M, \N$ subspaces, it hods that  $$P_{\M, \N}^*=P_{\N^{\perp},\M^{\perp}} \mbox{ and } \ol{P_{\M, \N}}=P_{\ol{\M}, \ol{\N}}.$$ Then, $P_{\M, \N} \in \MP(\HH)$ if and only if $\M$ and $\N$ are closed.
\end{prop}

\bigskip

We say that $T\in \lr(\HH,\KK)$ is \emph{decomposable} if $T$  admits the componentwise sum decomposition
\begin{equation} \label{Tdecom}
	T=T_0 \ \hat{\oplus} \ T_{\mul},
\end{equation} 
where  $T_0:=T \cap (\cldom T \times \cldom T^*)$ and $T_{\mul}:=\{0\} \times \mul T.$ When $T \in \lr(\HH)$ is closed, $\mul T=(\dom T^*)^{\perp}$ and $T$ is decomposable, see \cite{Hassi2}.
In this case, $T_{\mul}$ is a closed linear relation and the operator part $T_0$ is a densely defined  closed operator from $\cldom T$ to $\cldom T^*$  with 
$\dom T_0=\dom T$ and $\ran T_0 \subseteq  \cldom T^*.$ 

\begin{prop}\label{PyI-P} Let $P_{\M, \N}\in \MP(\HH).$ Then
	$P_{\M, \N}$ is decomposable and 
	$$P_{\M, \N}=P_{\M \ominus (\M\cap \N) \pl \N} \ \hat{\oplus} \ (\{0\} \times \M \cap \N).$$
\end{prop}

\begin{proof} Since $\M$ and $\N$ are closed,  $P_{\M, \N}$ is closed and then decomposable, and  $\M=\M \ominus (\M\cap \N) \oplus \M \cap\N.$ Then $Q:=P_{\M \ominus (\M\cap \N) \pl \N}$ is an operator with $\dom Q=\M \ominus (\M\cap \N) \dotplus \N=\M+\N=\dom \PMN$ and $\ran Q \subseteq (\M \cap\N)^\bot.$ Also, it is easy to check that $P_{\M, \N}\subseteq Q \ \hat{\oplus} \ (\{0\} \times \M \cap \N)$ and since $\mul(Q \ \hat{\oplus} \ (\{0\} \times \M \cap \N))=\M\cap \N,$ by Lemma \ref{lemalr}, the result follows. Since $\ran Q \subseteq \ol{\M^{\perp}+\N^{\perp}}=\cldom \PMN^*$, then $(P_{\M, \N})_0=Q.$ 
\end{proof}

\section{Matrix representation of multivalued projections} \label{three}
In what follows we characterize those multivalued projections admitting matrix representations with respect to the closure of their ranges. Our goal is to get a representation similar to the standard matrix representation of a bounded projection onto its range
\begin{equation}\label{mstandard}
	\matriz{I}{x}{0}{0},
\end{equation} see for example \cite{Ando}. Under this matrix representation, several properties of a projection can be obtained by studying the $(1,2)$-coefficient of the matrix, see for example 
\cite{And}. Even though the matrix representation of relations is not unique, a representation similar to that of  \eqref{mstandard} can be given when the range of the multivalued projection is closed.

\begin{lema}\label{lemarepSp} Let $\M, \N$ be subspaces of $\HH.$ Then the  following are equivalent:
	\begin{enumerate}
		\item $P_{\M,\N}$ admits a matrix representation with respect to $\overline{\M};$
		\item $P_{\overline{\M}}(\M+\N)\subseteq \M+\N;$
		\item $\M+\N=(\M+\ol{\M}\cap\N) \oplus (\M+\N)\cap \M^{\perp};$
		\item $P_{\M^{\perp}}(\N)=(I-P_\M)(\N).$
	\end{enumerate} 
\end{lema}

\begin{proof}	$i) \Leftrightarrow ii)$: From Theorem \ref{theoMLR}, $i)$ holds if and only if $P_{\overline{\M}}(\dom P_{\M, \N})\subseteq\dom P_{\M, \N}$ and $P_{\overline{\M}}(\mul P_{\M, \N})\subseteq\mul P_{\M, \N}.$ Since the second inclusion is automatic, the equivalence follows.
	
	$ii) \Leftrightarrow iii)$: If $ii)$ holds then $iii)$ follows from Lemma \ref{lemasubs}. The converse is straightforward.
	
	$ii) \Leftrightarrow iv)$: If $ii)$ holds then  $P_{\M^{\perp}}(\N) =P_{\M^{\perp}}(\M+\N)= (\M+\N)\cap \M^{\perp}.$ Given $w\in P_{\M^{\perp}}(\N)$ write $w=m+n$ for some $m\in \M$ and $n \in \N.$ Then $n=w-m\in \M^{\perp}\oplus \M =\dom P_\M$ and
	$m=-P_\M n$ or $w=(I-P_\M)n.$
	The other inclusion always holds.
	
	Conversely, if $iv)$ holds then  $P_{\overline{\M}}(\N)=(I-P_{\M^{\perp}})(\N)\subseteq \N+(I-P_\M)(\N) \subseteq \M+\N.$ Then $P_{\overline{\M}}(\M+\N)=\M+P_{\overline{\M}}(\N)\subseteq \M+\N.$
\end{proof}

In order to get a matrix representation for $\PMN$ we start by describing the $(1,2)$-coefficient given in \eqref{Acanon},
\begin{equation} \label{eqx0}
	x:=P_{\ol{\M}} \PMN|_{\M^{\perp}}=\PMN \cap (\M^{\perp} \times \HH).
\end{equation}

\begin{lema}\label{lemarepx0} If $\PMN$ admits a matrix representation with respect to $\ol{\M}$ then $\dom x=P_{\M^{\perp}}(\N)$ and
	\begin{equation*}
		x=\{(P_{\M^{\perp}}n,-P_\M n) : n \in \N\}=-P_\M ((I-P_\M)|_{\N})^{-1}.
	\end{equation*}
\end{lema}
\begin{proof} By  Lemma \ref{lemarepSp},  $P_{\M^{\perp}}(\N)=(I-P_\M)(\N)=(\M+\N) \cap \M^{\perp}=\dom x.$  To see the first equality, let $(u,v) \in x$ then $u=m+n \in \M^{\perp}$ for some $m\in \M$ and $n\in \N,$ and $v=m.$ Then $u=P_{\M^\bot}u=P_{\M^\bot}n$ and $n=(m+n)-m\in\M^{\perp}\oplus \M=\dom P_\M$ so that
	$P_\M n=-m.$ Hence $(u,v)=(P_{\M^\bot}n,-P_\M n).$ The other inclusion is straightforward.
	
	To prove the second equality, set $y:=-P_{\M}((I-P_\M)|_\N)^{-1}.$ Since $\ran ((I-P_\M)|_\N)^{-1}=\dom((I-P_\M)|_\N) \subseteq \dom P_\M,$ then $\dom y=\dom ((I-P_\M)|_\N)^{-1}=(I-P_\M)(\N)=\dom x.$ On the other hand, $\mul y= P_{\M} (\mul((I-P_\M)|_\N)^{-1})=P_{\M} (\ker(I-P_\M)|_\N)=P_{\M}(\M\cap \N)=\M \cap \N=\mul x.$ 
	Finally, $x \subseteq y.$ In fact, consider $(P_{\M^\bot} n, -P_{\M}n)$ for some $n \in \N \cap (\M\oplus \M^{\perp}).$ Since $P_{\M^{\perp}}|_{\M \oplus \M^{\perp}}=I-P_\M,$ if $n \in \M \oplus \M^{\perp}$ then $(n, P_{\M^{\perp}}n) \in (I-P_\M)|_{\N}$ so that $(P_{\M^{\perp}}n, n)  \in((I-P_\M)|_\N)^{-1}$ and  $(P_{\M^\bot} n, -P_\M n) \in -P_\M ((I-P_\M)|_\N)^{-1}=y.$ Hence, by Lemma \ref{lemalr}, $x=y.$	\end{proof}	
By Proposition \ref{lemarepSp}, if $\M$ is  closed then $\PMN$ always admits a matrix representation with respect to $\M$ similar to that of a bounded projection.

\begin{prop}\label{repSpMcerrado} Let $\M$ be a closed subspace of $\HH.$ Then
	\begin{equation}\label{PMNmatrixclosed}
		P_{\M,\N}=\matriz{I}{x}{0}{0}_{\M}
	\end{equation}
	where $x=-P_\M (P_{\M^{\perp}}|_\N)^{-1}.$
\end{prop}
\begin{proof}  Let $E$ be the linear relation generated by the block matrix \eqref{PMNmatrixclosed}. Since $\dom x=P_{\M^{\perp}}(\N)$ and $\mul x=\M \cap \N,$  $\dom E=\M\oplus\dom x=\M+\N$  and $\mul E=\mul x=\M \cap \N.$
	
	Finally, if $(u,v)\in E$ then there exist $m\in\M$ and $n\in \N$ such that $(u,v)=(m+P_{\M^\bot}n,m-P_{\M}n).$ Then $((m-P_{\M}n)+n, m-P_{\M}n)\in P_{\M,\N}$ and $E\subseteq P_{\M,\N}.$ Hence, by Lemma \ref{lemalr}, $P_{\M,\N}=E.$
\end{proof} 

\begin{cor} \label{estrellaclausura}
	Given $\M, \N$ subspaces, it hods that 
	\begin{equation} \label{MatrixPP}
		\ol{\PMN}=\matriz{I}{x}{0}{0}_{\ol{\M}}, \ \ \ 	\PMN^*=\matriz{I}{0}{x^*}{0}_{\ol{\M}} 
	\end{equation}
	where $x=-P_{\ol{\M}} (P_{\M^{\perp}}|_{\ol{\N}})^{-1}.$
\end{cor}

\begin{proof} By Proposition \ref{L2},  $\ol{E}=P_{\ol{\M},\ol{\N}}.$ Then, by Proposition \ref{repSpMcerrado}, $\ol{E}$ has the matrix representation given in \eqref{MatrixPP}. Finally, since $E^*=\ol{E}^*,$ using Proposition \ref{starMR}, we get the matrix representation for $E^*.$
\end{proof}

\begin{cor}
	
	$\PMN$ is closed if and only if $P_{\M,\N}=\matriz{I}{x}{0}{0}_{\ol{\M}}$ with $x$  closed.
\end{cor}

\begin{proof} Suppose that $\PMN$ is closed. Then, by Proposition \ref{L2}, $\M$ and $\N$ are closed. So that, by Proposition \ref{repSpMcerrado}, $P_{\M,\N}=\matriz{I}{x}{0}{0}_{\M}$ with $x=\PMN \cap (\M^{\perp} \times \HH)$ closed.
	The converse follows by items $2$ and $1$ of Lemma \ref{rcclosed}.
\end{proof}

Our next goal is to get a matrix representation with respect to $\ol{\M}$  of every  representable $P_{\M, \N}$. To this end, we need the following lemma.

\begin{lema}\label{lemasp} It holds that
	$$P_{\M,\N}=P_{\M,\ol{\M}\cap\N}P_{\ol{\M},\N}.$$
\end{lema}
\begin{proof} Set $E:=P_{\M,\ol{\M}\cap\N}P_{\ol{\M},\N}.$ Since $I_\M \subseteq P_{\M,\ol{\M}\cap\N},$ it follows that  $P_{\M,\N}\subseteq E.$ Conversely, if $(u,v) \in E$ then there exists $ w \in \HH$ such that $(u,w)=(m+n,m)$ with $m\in \ol{\M}$ and $n \in \N,$ and $(w,v)=(m'+n',m')$ with $m' \in \M$ and $n' \in \ol{\M} \cap \N.$ Hence $u=m+n=w+n=m'+(n'+n)\in \M+\N$ and $(u,v)=(m'+(n'+n),m') \in \PMN.$ Then $E \subseteq \PMN.$
\end{proof}

Using Proposition \ref{repSpMcerrado} and Lemma \ref{lemasp} we get the following matrix representation onto $\ol{\M}$ for any representable multivalued projection $\PMN.$

\begin{thm}\label{repSp} If $\PMN$ admits a matrix representation with respect to $\ol{\M}$ then
	\begin{equation}\label{PMNmatrix}
		P_{\M,\N}=\matriz{P_{\M, \ol{\M}\cap\N}}{x}{0}{0}_{\ol{\M}}
	\end{equation}
	where $x$ is as in \eqref{eqx0}.
\end{thm}
\begin{proof} 
	From  Lemma \ref{lemasp}, $P_{\M,\N}=P_{\M,\ol{\M}\cap\N}P_{\ol{\M},\N}.$
	It can be easily seen that $$P_{\M, \ol{\M}\cap\N}=\matriz{P_{\M, \ol{\M}\cap\N}}{0}{0}{0}_{\ol{\M}}.$$ By Proposition \ref{repSpMcerrado}, $$P_{\ol{\M}, \N}=\matriz{I}{x_0}{0}{0}_{\ol{\M}},$$ where $x_0:=P_{\ol{\M},\N}\cap (\M^{\perp} \times \HH).$ 
	
	By Lemma \ref{lemasp} and \eqref{Productres}, 
	$$x=P_{\M,\ol{\M}\cap\N}(P_{\ol{\M},\N}\cap (\M^{\perp} \times \HH))=P_{\M,\ol{\M}\cap\N}x_0.$$
	By Lemma \ref{producto},
	\begin{align*}
		P_{\M,\N}&=\begin{pmatrix}
			P_{\M, \ol{\M}\cap\N}\\0
		\end{pmatrix} \begin{pmatrix}
			I & x_0
		\end{pmatrix}\supseteq {\begin{pmatrix}
				\begin{pmatrix}
					P_{\M, \ol{\M}\cap\N}\\0
				\end{pmatrix} I &  \begin{pmatrix}
					P_{\M, \ol{\M}\cap\N}\\0
				\end{pmatrix} x_0
		\end{pmatrix}}_{\ol{\M}}
		\\
		&=\matriz{P_{\M, \ol{\M}\cap\N}}{x}{0}{0}_{\ol{\M}}:=E,
	\end{align*}
	where the last equality follows using Lemma \ref{filacolumna}.
	
	Since $\dom x=P_{\M^{\perp}}(\N)$ and $\mul x=\M \cap \N,$ it comes that $\dom E=(\M+\ol{\M}\cap\N) \oplus \dom x= \M+\N$
	and $\mul E=\ol{\M}\cap\N \cap \M + \mul x=\M \cap \N.$ Then \eqref{PMNmatrix} follows from Lemma \ref{lemalr}. 
\end{proof} 

\begin{obs} From the proof of Theorem \ref{repSp} it follows  that if $x=\PMN \cap (\M^{\perp} \times \HH)$ then
	$$x=P_{\M,\ol{\M}\cap\N}x_0$$ where $x_0=P_{\ol{\M},\N} \cap (\M^{\perp} \times \HH).$ In other words, if $\M_0$ is a closed subspace of $\HH$ then for any $\M$ such that $\ol{\M}=\M_0,$ the $(1,2)$-coefficient of the matrix representation of $\PMN$ given in Theorem \ref{repSp} is parameterized in terms of the fixed $(1,2)$-coefficient $x_0$ of the matrix representation of $P_{\M_0,\N}.$
\end{obs}

As a corollary we get a version for multivalued projections of Ando's matrix representation of a projection.

\begin{thm} [\normalfont{cf. \cite[Theorem 2.6]{Ando}}]\label{Ando}
	If $\PMN$ admits a matrix representation with respect to $\ol{\M}$ then
	\begin{equation}
		P_{\M,\N}=\matriz{P_{\M, \ol{\M}\cap\N}}{-P_{\M}((I-P_\M)|_\N)^{-1}}{0}{0}_{\ol{\M}}.\nonumber
	\end{equation}
\end{thm}

\begin{proof} The result follows from Theorem \ref{repSp} and Lemma \ref{lemarepx0}.
\end{proof}

\begin{cor} \label{semiMR} The multivalued projection $\PMN$ admits the matrix representation 
	\begin{equation} \label{MatrixP}
		P_{\M,\N}=\matriz{I_\M}{x}{0}{0}_{\ol{\M}}
	\end{equation}
	where  $x$ is as in \eqref{eqx0} if and only if $\N \subseteq \M \oplus \M^{\perp}.$
\end{cor}
\begin{proof} If \eqref{MatrixP}  holds then $\N\subseteq \dom P_{\M,\N}\subseteq \M \oplus \M^{\perp}.$ Conversely,  if  $\N \subseteq \M \oplus \M^{\perp}$ then $P_{\ol{\M}}(\M+\N) \subseteq \M.$ So that, by Theorem \ref{repSp} and the fact that $\ol{\M} \cap \N=\M \cap \N,$ \eqref{MatrixP} follows,
	because $P_{\M, \M \cap \N}=I_\M \ \hat{+} \ ((\M\cap \N) \times \{0\})=I_\M \ \hat{+} \ (\{0\} \times (\M\cap \N))$ and $\mul x= \M \cap \N.$ Then $\matriz{P_{\M, \M \cap \N}}{x}{0}{0}_{\ol{\M}}=\matriz{I_\M}{x}{0}{0}_{\ol{\M}}.$
\end{proof}

\begin{nota}
	Any $x \in \lr(\M^{\perp},\ol{\M})$ can be seen as the subspace of $\HH$ whose elements are $u+v$ under the correspondence $(u,v)  \leftrightarrow u+v.$ 
\end{nota}

\begin{prop} \label{prop27}
	Let $\W$ be a subspace of $\ol{\M}$ and $x \in \lr(\M^{\perp},\ol{\M})$  with $\ran x\subseteq \M.$ If $$E=\matriz{P_{\M, \W}}{x}{0}{0}_{\ol{\M}}$$ then $E=P_{\M,-x+\W}.$
\end{prop}
\begin{proof}
	Since $(y,z) \in E$ if and only if $(y,z)=(m+w+u,m+v),$ where $m \in \M, w\in\W$ and $(u,v) \in x,$ then
	$$(y,z)=((m+v)+(w+u-v),m+v),$$ with $m+v \in \M.$ Thus, $(y,z)\in P_{\M,\N}$ where $\N=\{u-v: (u,v)\in x\}+\W=-x+\W,$ i.e., $E\subseteq P_{\M,\N}.$ 
	Since $\ran x \subseteq \M,$ then $\M+\N\subseteq \M+\W+\dom x=\dom E.$ Let us see that $\M\cap \N\ \subseteq \mul E.$ In fact, if $z \in \M \cap \N \subseteq \dom E$ then $z=w+u-v$ for some $w \in \W$ and $(u,v) \in x.$ Then $u=z-w+v \in \dom x \cap \ol{\M} \subseteq \M^{\perp} \cap \ol{\M}.$ Then $u=0,$ $z=w-v,$ $v \in \mul x \subseteq \M$ and $w \in \W \cap \M.$ Therefore $z \in  \W \cap \M+ \mul x=\mul E.$
	
	Hence, by Lemma \ref{lemalr}, $E=P_{\M,\N}.$
\end{proof}

\begin{cor} The class of multivalued projections with range $\M$ admitting matrix representation with respect to $\ol{\M}$ is given by 
	$$\left\{\matriz{P_{\M, \W}}{x}{0}{0}_{\ol{\M}}: \W \mbox{ is a subspace of } \ol{\M} \mbox{ and } x \in \lr(\M^{\perp},\ol{\M}), \ran x \subseteq \M \right\}.$$
\end{cor}

\subsection{Matrix representation of super-idempotents} \label{third1}

A linear relation $E \subseteq \HH \times \HH$ is called  \emph{sub-idempotent} if $E^2 \subseteq E,$  \emph{super-idempotent} if $E \subseteq E^2$ and an  \emph{idempotent} if the equality holds.
These classes of relations were studied in detail in \cite{idem}. Here we focus on the matrix representation of super-idempotents  and we get as a corollary a result concerning the representation of idempotents. We omit the study of the matrix representation of sub-idempotents since it follows similar ideas.

The next results gather some properties and descriptions of the class of super-idempotent linear relations, see \cite{idem}.
\begin{lema} \label{subsuper} The set of super-idempotent relations is given by
	$$\{P_{\M,\N} \ \hat{+} \ (\{0\} \times \St) : \M, \N, \St \mbox{ are subspaces of } \HH\}.$$
\end{lema}

\begin{lema} \label{propmul}  \label{cormul2} Let $T:=\PMN \ \hat{+}  \ (\{0\} \times \St)$. Then $\dom T=\M+\N,$ $\ran T=\M+\St,$  $\ker T=\N+ \M \cap \St$ and $\mul T= \St + \M \cap \N.$
\end{lema}

Super-idempotents admit many different representations, amongst them we distinguish the ones given in the following lemma, which we call the \emph{canonical representations}.
\begin{lema} \label{canonical} Let $E\in \lr(\HH).$ Then
	$E$ is super-idempotent if and only if $E=P_{\ker(I-E), \ker E} \ \hat{+} \ (\{0\}\times \mul E).$
\end{lema}

From the matrix representation of multivalued projections we can easily get the matrix representation for super-idempotents. 

\begin{lema} \label{propdecomRS2} Let $ T\in \lr(\HH)$ be super-idempotent. If $T=\PMN \ \hat{+}  \ (\{0\} \times \St)$ with $\M=\ker(I-T)$, $\N=\ker T$ and $\St=\mul T$ then
	$T$ admits a matrix representation with respect to $\ol{\M}$ if and only if  \begin{equation} 
		\M+\N= (\M + \ol{\M}\cap \N) \oplus  (\M+\N)\cap \M^{\perp} \ \text{and} \ \St=\St \cap \ol{\M} \oplus\St \cap \M^{\perp}.\nonumber
	\end{equation}
\end{lema}

\begin{proof} It follows from Theorem \ref{theoMLR} and Lemmas \ref{propmul} and \ref{lemasubs}.
\end{proof}

From now on, given a super-idempotent $T$ we consider the canonical representation given in Lemma \ref{canonical}, i.e., $$T:=\PMN \ \hat{+}  \ (\{0\} \times \St)$$ where $\M=\ker(I-T)$, $\N=\ker T,$ and $\St=\mul T.$ 

If a super-idempotent $T$ admits a matrix representation with respect to $\ol{\M}$ then $\M+\N= (\M + \ol{\M}\cap \N) \oplus  (\M+\N)\cap \M^{\perp},$ so that by Theorem \ref{repSp}, $P_{\M,\N}$ admits a matrix representation with respect to $\ol{\M}.$ On the other hand, $\St=\St \cap \ol{\M}\oplus\St \cap \M^{\perp},$ then $\HH \times \St$ admits a matrix representation with respect to $\ol{\M}.$ Finally,  the sum of both matrix representations gives the matrix representation of $T$ as shows the following result.

\begin{prop} \label{Tsuper} Let $T\in \lr(\HH)$ be super-idempotent. If $T$ admits a matrix representation with respect to $\ol{\M}$ then 
	\begin{equation} 
		T=\matriz{P_{\M ,\ol{\M} \cap \N}}{x \ \hat{+} \ (\{0\}\times \St \cap \ol{\M})}{0 \ \hat{+} \ (\{0\}\times \St \cap \M^{\perp})}{0}_{\ol{\M}}\nonumber
	\end{equation}
	where $x =P_{\M,\N}\cap (\M^\bot\times\HH).$
\end{prop}

\begin{proof} It holds that $T\subseteq P_{\M, \N} + (\HH \times \St),$ and both linear relations have the same domain and multivalued part, because $\St=\mul T.$ Then $T=P_{\M, \N} + (\HH\times \St).$
	On the one hand, applying Theorem \ref{repSp},
	$$P_{\M,\N}=\matriz{P_{\M, \overline{\M} \cap \N}}{x}{0}{0}_{\ol{\M}}$$
	where $x=P_{\M,\N}\cap (\M^\bot\times\HH).$ On the other hand, if $\St_1:=\St \cap \ol{\M}$ and $\St_2:=\St \cap \M^{\perp}$ then $\St=\St_1\oplus\St_2$ and
	$$\HH \times \St=\matriz{0 }{\M^{\perp} \times \St_1 }{\ol{\M}\times \St_2}{0}_{\ol{\M}}.$$
	Then, by Lemma \ref{suma},  
	\begin{align*}
		T&=\matriz{P_{\M, \N\cap\overline{\M}}}{x+\M^{\perp} \times \St_1}{\ol{\M} \times \St_2}{0}_{\ol{\M}}=\matriz{P_{\M, \N\cap\overline{\M}}}{x \ \hat{+} \ (\{0\}\times \St_1)}{0 \ \hat{+} \ (\{0\}\times \St_2)}{0}_{\ol{\M}}.
	\end{align*}
\end{proof}

\begin{cor} \label{semiMR2} Let $T  \in \lr(\HH)$ be super-idempotent. Then $T$ admits the matrix representation 
	\begin{equation} \label{MatrixidemP2}
		T=\matriz{I_\M}{x \ \hat{+} \ (\{0\}\times \St \cap \ol{\M})}{0 \ \hat{+} \ (\{0\}\times \St \cap \M^{\perp})}{0}_{\ol{\M}}
	\end{equation}
	where $x =P_{\M,\N}\cap (\M^\bot\times\HH)$ if and only if $\N \subseteq \M \oplus \M^{\perp}.$ 
\end{cor}
\begin{proof} As in Corollary \ref{semiMR}, if $\N \subseteq \M \oplus \M^{\perp}$ then $P_{\ol{\M}}(\M+\N) \subseteq \M.$ Then, by Proposition \ref{Tsuper} and the fact that $\ol{\M} \cap \N=\M \cap \N,$ \eqref{MatrixidemP2} follows. The converse is straightforward.
\end{proof}

\begin{prop} \label{propMidem2} Given a subspace $\M$ of $\HH,$ let $\W, \St_1$ be subspaces of $\ol{\M},$ $\St_2$ be a subspace of $\M^{\perp}$ and $x\in \lr(\M^{\perp},\ol{\M})$  with $\ran x \subseteq \M.$ Define  $$E:=\matriz{P_{\M, \W}}{x \ \hat{+} \ (\{0\}\times \St_1)}{0 \ \hat{+} \ (\{0\}\times \St_2)}{0}_{\ol{\M}}.$$ Then $E$ is super-idempotent.
	
\end{prop}
\begin{proof} It holds that
	$E=\matriz{P_{\M, \W}}{x}{0}{0}_{\ol{\M}}+\matriz{0 }{\M^{\perp} \times \St_1 }{\ol{\M}\times \St_2}{0}_{\ol{\M}}.$ Then, by Proposition \ref{prop27}, 
	$$E=P_{\M,\N} \ \hat{+} \ (\{0\} \times \St),$$ where $\N:=\W+(-x)$ and $\St:=\St_1 \oplus \St_2.$ By Lemmas \ref{subsuper} and \ref{propmul}, $E$ is super-idempotent with canonical representation 
	$E=P_{\M+\N\cap \St, \N+\M\cap\St, \St+\M\cap\N}.$
\end{proof}

\begin{cor} \label{corMidem2id} Given $\M$ a subspace of $\HH,$ $\St_1$ and  $\St_2$ subspaces of $\ol{\M}$ and $\M^{\perp},$ respectively, and $x\in 
	\lr(\M^{\perp},\ol{\M})$ with $\ran x \subseteq \M.$ Consider 
	\begin{equation} \label{Ematrixsuper}
		E:=\matriz{I_\M}{x \ \hat{+} \ (\{0\}\times \St_1)}{0 \ \hat{+} \ (\{0\}\times \St_2)}{0}_{\ol{\M}}. 
	\end{equation}
	Then $E$ is super-idempotent with canonical representation 
	\begin{equation} \label{ECanonicalR}
		E=P_{\M\oplus \St_2 \cap \dom x, -x+\St_1\cap \M, (\St_1+\mul x)\oplus \St_2}.
	\end{equation}
\end{cor}

\begin{proof} Applying Proposition \ref{propMidem2} with $\W=\{0\},$ it holds that $E=P_{\M,\N} \ \hat{+} \ (\{0\} \times \St)$ is super-idempotent with $\St=\St_1 \oplus \St_2$ and $\N=\{u-v: (u,v) \in x\}.$ 
	
	By Lemma \ref{propmul}, the canonical representation of $E$ is $$E=P_{\M+\N\cap \St, \N+\M\cap\St, \St+\M\cap\N}.$$ But $\M \cap \N=\mul x$ so that $\St+\M\cap \N=(\St_1+\mul x) \oplus \St_2;$ and $\N+\M\cap\St=\N+\M \cap \St_1.$ Finally, $\M+\N\cap \St=\M \oplus \St_2 \cap \dom x.$ In fact, $z \in \M+\N \cap \St$ if and only if $z=m+u-v,$ where $m \in \M$ and $(u,v) \in x \cap (\St_2 \times \St_1)$ if and only if $z=m'+u,$ where $m' \in \M$ and $u \in \dom x \cap \St_2.$ Finally, $\N=-x$ by definition.
\end{proof}

\begin{thm} Let $E$ be as in \eqref{Ematrixsuper}. Then $E$ is idempotent if and only if $$x(\St_2) \subseteq \St_1+\mul x.$$
\end{thm}
\begin{proof} By Corollary \ref{corMidem2id}, $E$ admits the canonical representation \eqref{ECanonicalR}. Write $\M':=\M\oplus \St_2 \cap \dom x,$ $\N':= -x+\St_1\cap \M$ and $\St':=(\St_1+\mul x)\oplus \St_2.$ Then, by \cite[Proposition 4.9]{idem}, $E$ is idempotent if and only if $(\M'+\N')\cap \St'=\M' \cap \N'.$ 
	
	Since $\M'+\N'=\M+(-x)=\M\oplus \dom x$ then, using that $\mul x \subseteq \M,$
	\begin{align*}
		(\M'+\N') \cap \St'&=(\M \oplus \dom x) \cap ((\St_1+\mul x)\oplus \St_2) \\
		&=(\M \cap \St_1+\mul x) \oplus \dom x \cap \St_2.
	\end{align*}
	On the other hand, 
	\begin{align*}
		\M' \cap \N'&=(\M \oplus \St_2 \cap \dom x) \cap (-x+\M \cap \St_1)\\
		&=\{u-v : (u,v) \in x|_{\St_2}\} + \M \cap \St_1.
	\end{align*}
	Then  $(\M'+\N')\cap \St'=\M' \cap \N'$ if and only if 
	$$(\M \cap \St_1+\mul x) \oplus \dom x \cap \St_2= \{u-v : (u,v) \in x|_{\St_2}\} + \M \cap \St_1.$$ 
	
	We claim that the last equality holds if and only if $x(\St_2) \subseteq \St_1+\mul x.$ In fact, if the equality holds, then $x(\St_2) \subseteq \M \cap \St_1+\mul x.$ Conversely, if $x(\St_2) \subseteq \St_1+\mul x$ then the inclusion $\supseteq$ is straightforward because $\ran x \subseteq \M.$ To see the opposite inclusion we only need to check that $\dom x \cap \St_2 \subseteq \M'\cap \N',$ since $\mul x \subseteq -x(\St_2)$ trivially. But if $u \in \dom x \cap \St_2$ then there exists $v=v_1+v_2,$ $v_1 \in \St_1 \cap \M$ and $v_2 \in \mul x$ such that $(u,v) \in x.$ Then $(0,v_2) \in x,$ $(u,v_1)=(u,v)-(0,v_2) \in x,$ $u=u-v_1+v_1 \in -x|_{\St_2}+ \M \cap \St_1$ and then the equality holds.
\end{proof}

\section{W-selfadjoint multivalued projections}  \label{fourth}
Denote by $L(\HH)^s$  the set of selfadjoint operators in $L(\HH)$ and  $L(\HH)^+$ the cone of positive semi-definite  bounded linear operators on $\HH.$ Along this section $W\in L(\HH)^s$ and $\St$ is a subspace of $\HH.$

Consider the sesquilinear form
$$\PI{x}{y}_W:=\PI{Wx}{y} \mbox{ for } x, y \in \HH.$$

\begin{Def} The $W$-\emph{orthogonal companion} of $\St$ is the subspace
	$$\St^{\perp_W}:=\{x \in \HH: \PI{x}{s}_W=0 \mbox{ for every } s \in 
	\St\}.$$ 
\end{Def}
It is easy to check that if $W \in L(\HH)^s$ then $\St^{\perp_W}=(W\St)^{\perp}=W^{-1}(\St^{\perp})=(\ol{\St})^{{\perp}_W}.$ In general,
$\St + \St^{\perp_W} \subsetneq \HH$
and the sum may not be direct.

Given $\St$ a subspace of $\HH,$ consider the  multivalued projection 
\begin{equation}\label{pas18}
	\pas:=P_{\St,\St^{\perp_W}}.
\end{equation}

\medskip

The following lemma collects some basic properties of $\pas$.
\begin{lema}\label{propPas} The following properties hold:
	\begin{enumerate}
		\item[1. ] $\dom \pas=\St+(W\St)^\bot$ and $\mul \pas=\St\cap (W\St)^\bot.$ If $W \in L(\HH)^+$ then $\mul \pas=\St \cap \ker W.$
		\item[2. ] $\pas^*=P_{\overline{W\St}, \St^\bot},$ $\ol{\pas}=P_{W, \ol{\St}}.$ Then,  $\pas$ is closed if and only if $\St$ is closed.
		\item[3. ] $\pas$ is $W-$symmetric, i.e., $W\pas\subseteq (W\pas)^*.$
	\end{enumerate}
	
\end{lema}
\begin{proof} The first part of item $1$ follows from Proposition \ref{L1}. If $W \in L(\HH)^+$ and $x \in \St\cap (W\St)^\bot$ then $W^{1/2}x \in W^{1/2} \St\cap (W^{1/2}\St)^\bot=\{0\}$ and $x \in \St \cap \ker W.$ The other inclusion always holds.
	
	Item $2$ follows from Proposition \ref{L2}.
	
	$3$: Since $W\pas=\{(s+n,Ws): s\in\St, n\in (W\St)^\bot\},$ if $(s+n,Ws)\in W\pas$ then, by item $2,$ $Ws \in \ran \pas^*$ and $Wn \in W(W^{-1}(\St^{\perp})) \subseteq \St^{\perp}=\ker \pas^*,$ so that $(W(s+n),Ws) \in \pas^*.$ Therefore, $(s+n,Ws)\in \pas^* W=(W\pas)^*.$
\end{proof}

\begin{Example} Let $W$ be a selfadjoint symmetry, so that  $W=W^*=W^{-1},$ and consider the indefinite inner product  space $\K{x}{y}:=\PI{Wx}{y}.$ Then $(\HH, \K{ \cdot }{ \cdot })$ is a {\emph{Krein space}} \cite{Azizov}. Given a closed subspace $\St,$ if $\St^{[\perp]}:=\St^{{\perp}_W},$ 
	the \emph{isotropic part} of $\St$ is $\St^{\circ}:=\St \cap \St^{[\perp]}.$ Then $\St$ is \emph{non-degenerate} if $\St^{\circ}=\{0\},$ or equivalently $\ol{\St \ [\dotplus] \ \St^{[\perp]}}=\HH,$ where $[+]$ stands for the $[ \cdot, \cdot ]$-orthogonal sum and we write $[\dotplus]$ when it is direct. $\St$ is \emph{pseudo-regular} if $\St \ [+] \ \St^{[\perp]}$ is closed and $\St$ is \emph{regular} if $\St \ [\dotplus] \ \St^{[\perp]}=\HH.$ Regular subspaces are non-degenerate, but pseudo-regular subspaces can be degenerate. 
	It is straighforward to establish the following correspondence:
	\begin{itemize}
		\item[-] $\St$ is non-degenerate if and only if $P_{W,\St}$ is a projection.
		\item[-]  $\St$ is pseudo-regular if and only if $P_{W,\St}$ is closed with closed domain. 
		\item[-]  $\St$ is regular if and only if $P_{W, \St} \in L(\HH).$
	\end{itemize}
	
	In this example we see that, even if the weight is invertible, we can have that $P_{W,\St}$ is not a projection, for example when $\St$ is pseudo-regular and degenerate.
\end{Example}

\subsection{Complementability and Schur complement} \label{fourth1}

The notion of com\-ple\-men\-ta\-bi\-li\-ty of an operator $W \in L(\HH)$ with respect to two given closed subspaces $\St$ and $\T$ of $\HH$ was studied for matrices by Ando \cite{AndoC} and extended to operators in Hilbert spaces by Carlson and Haynsworth \cite{Carlson}, by Corach et al.  \cite{Szeged} and Antezana et al. \cite{bilateral}. We use these ideas for $\St = \T$ and $W \in L(\HH)^s.$ Complementability and quasicomplementability of $W$ with respect to $\St$ \cite{Quasi, Szeged} in terms of $\pas$ read as follows.

\begin{Def} We say that $W$ is $\St-$\emph{quasicomplementable} if $\cldom \pas=\HH$ and is $\St-$\emph{complementable} if $\dom \pas=\HH.$
\end{Def}

\begin{lema} \label{lemaWS} If $W$ is $\St$-quasicomplementable then $$\mul \pas=\St\cap \ker W.$$
\end{lema}
\begin{proof} It always holds that $\St \cap \ker W \subseteq  \St \cap W^{-1}(\St^{\perp}) =\mul \pas.$ On the other hand, since $\HH=\cldom \pas$ it follows that $\{0\}=\St^{\perp} \cap \ol{W\St}.$ Let $x \in \St \cap W^{-1}(\St^{\perp})$ then $Wx \in W\St \cap \St^{\perp} \subseteq \ol{W\St} \cap \St^{\perp}= \{0\}$ and $x \in \St \cap \ker W.$
\end{proof}

\begin{obs} \begin{enumerate}
		\item[1.] If $W$ is $\St$-complementable then $\mul \pas \subseteq \ker W$ and $W\pas$ is an operator.  If $W\in L(\HH)^+$ then, by Lemma \ref{propPas}, $W^{1/2}\pas$ is always an operator. 
		\item [2.] If $W$ is $\St$-complementable then $W$ is $\ol{\St}$-complementable. If in addition $\ker W=\{0\}$ then $\HH=\St \dotplus W^{-1}(\St^\perp) =\ol{\St} \dotplus W^{-1}(\St^\perp).$ Hence
		$\St=\ol{\St}$ and $\pas$ is an operator. 
	\end{enumerate}
\end{obs}

\begin{prop} \label{bounded} The following statements hold:
	\begin{enumerate}
		\item[1. ] $W$ is $\St$-quasicomplementable if and only if $\pas^*$ is a projection.
		\item[2. ] If $W$ is $\St$-complementable then $\pas^*$ is a bounded projection. If $\St$ is closed the converse follows.
	\end{enumerate}
\end{prop}
\begin{proof}
	$1$: follows from the fact that $\St^\bot\cap \overline{W\St}=\{0\}$ if and only if $\overline{\St+(W\St)^\bot}=\HH.$  
	
	$2$: If $\St+(W\St)^\bot=\HH$ then $\ol{\St}+(W\St)^\bot=\HH$ and $\St^\bot \dot{+}\overline{W\St}$ is closed (see Proposition \ref{Deutch}). Then $\pas^*$ is a closed operator with closed domain and then, by the closed graph theorem, $\pas^*$ is a bounded operator. Finally, suppose that $\St$ is closed and $\pas^*$ is a bounded projector. Then $\dom \pas^* =\St^\bot \dot{+}\overline{W\St}$ is closed. Hence, $\HH=\St+(W\St)^\bot.$
\end{proof}

\begin{prop} \label{propScomp} The following are equivalent:
	\begin{enumerate}
		\item $W$ is $\St-$complementable;
		\item $W\pas=\pas^*W$ and $\ran W \subseteq \dom \pas^*;$
		\item There exists a $W$-selfadjoint $Q \in \Sp(\HH)$ with $\ran Q=\St$ and $\dom Q=\HH.$
	\end{enumerate}
\end{prop}
\begin{proof}$i)\Rightarrow ii)$: By Lemma \ref{propPas}, $W\pas\subseteq \pas^* W.$ If $W$ is $\St$-complementable then $\dom W\pas=\HH.$ Also, by Proposition \ref{bounded}, $\pas^*$ is a projection.  Then $\mul \pas^* W=\{0\}$ and equality follows by Lemma \ref{lemalr}. 
	
	Finally, from $\dom \pas^* W=\HH$ it follows that $\ran W \subseteq \dom \pas^*.$
	
	$ii)\Rightarrow iii)$: Since $\ran W \subseteq \dom \pas^*,$ we get that $\dom W\pas=\HH=\dom \pas.$ Taking $Q=\pas,$ $iii)$ follows.
	
	$iii)\Rightarrow i)$: Let $Q=P_{\St,\N}$ with $\dom Q=\HH=\St+\N$ and $WQ=Q^*W.$ We claim that $\N\subseteq (W\St)^\bot. $ In fact, if $n\in \N$ then $(n,0)\in WQ=Q^*W.$ Thus, $(Wn,0)\in Q^*$, or $Wn \in \ker Q^*=\St^{\perp},$ so that $n\in W^{-1}(\St^\bot)=(W\St)^\bot.$ Therefore, $\HH=\St+(W\St)^\bot.$
\end{proof}

\begin{obs} By the discussion in the proof of the above proposition,  if $W$ is $\St$-complementable then $\pas$ is the maximal multivalued projection with range $\St$ which is $W$-selfadjoint, that is, if $Q=P_{\St,\N}$ and $WQ=Q^*W$ then $Q\subseteq \pas.$
	
	In \cite{Szeged} several characterizations of complementability in terms of $W$- selfadjoint projections in $L(\HH)$ are studied. Unlike the multivalued projection $\pas,$ which always exists, these projections exist if and only if $W$ is $\St$-complementable, and by definition, they are bounded operators acting on $\HH.$ In fact, when $W$ is $\St$-complementable, all of these $W$-selfadjoint projections onto $\St$ are included in $\pas$ and both concepts coincide only when $W$ is invertible. 
\end{obs}

To get the matrix representation of $\pas$ when $\St$ is closed and $W$ is $\St$-complementable, consider the matrix representation of $W$
\begin{equation} \label{WMR}
	W=\matriz{a}{b}{b^*}{c}_{\St}\in L(\HH)^s.
\end{equation}

By Proposition \ref{repSpMcerrado},
\begin{equation}\label{pasRM}
	\pas=\matriz{I}{x}{0}{0}_{\St},
\end{equation}
where $x=-P_\St (P_{\St^{\perp}}|_{(W\St)^\bot})^{-1}.$

Since $\ker a = \ker (P_\St W|_\St)=\St \cap \ker P_\St W=\St \cap W^{-1}(\St^{\perp})=\mul \pas=\mul x,$ then
\begin{equation} \label{mulx}
	\mul x = \ker a.
\end{equation}

\begin{prop} \label{propScomp2} Let $\St$ be a closed subspace of $\HH,$ and $W$ and $\pas$ with matrix representations \eqref{WMR} and \eqref{pasRM}, respectively.
	Then  $W$ is $\St$-complementable if and only if $ax=b.$
	
	In this case, \begin{equation} \label{pascomp}
		\pas=\matriz{I}{a^{-1}b}{0}{0}_{\St}.
	\end{equation}
\end{prop}
\begin{proof} In what follows all matrix representations are with respect to $\St.$ By Lemma \ref{producto},
	\begin{align*}
		W\pas=\matriz{a}{b}{b^*}{c}\matriz{I}{x}{0}{0}=\begin{pmatrix} a \\ b^*
		\end{pmatrix} \begin{pmatrix} I & x \end{pmatrix}=\begin{pmatrix} \begin{matrix} a \\ b^* 
			\end{matrix} &  \begin{pmatrix} a \\ b^*
			\end{pmatrix} x \end{pmatrix},
	\end{align*}
	where we used that $\begin{pmatrix} a \\ b^*
	\end{pmatrix} \in L(\St, \HH)$ and Lemma \ref{filacolumna2}.
	But, by \eqref{mulx},  $\mul x \cap \ker a + \mul x \cap \ker b^*=\ker a =\mul x.$ Hence, by Lemma \ref{filacolumna},
	$$W\pas=\begin{pmatrix} a & ax \\ b^* & b^*x
	\end{pmatrix}.$$
	
	By Corollary \ref{estrellaclausura} and Proposition \ref{producto}, 
	\begin{align*}
		\pas^*W=\matriz{I}{0}{x^*}{0}\matriz{a}{b}{b^*}{c}=\begin{pmatrix} I \\ x^*
		\end{pmatrix} \begin{pmatrix} a & b \end{pmatrix}=\begin{pmatrix} \begin{matrix} a & b
			\end{matrix} \\  x^* \begin{pmatrix} a & b
		\end{pmatrix} \end{pmatrix},
	\end{align*}
	where we used that $\begin{pmatrix} a & b \end{pmatrix} \in L(\HH,\St)$ and Lemma \ref{filacolumna}.
	
	If $W$ is $\St-$complementable, by Proposition \ref{propScomp}, $W\pas=\pas^*W \in L(\HH)$ then $ax=b.$ The converse follows using \cite[Proposition 3.3]{Szeged}, where it was proved that $\St+(W\St)^{\perp}=\HH$ if and only if $\ran b \subseteq \ran a.$
	
	Clearly $x \subseteq a^{-1}b.$ But, $\dom x=\dom ax=\dom b= \St^{\perp}=\dom a^{-1}b,$ because $\ran b \subseteq \ran a$ and, by \eqref{mulx}, $\mul x =\ker a=\mul a^{-1}b.$ Then $x=a^{-1}b.$
\end{proof}

Given $W\in L(\HH)^{+}$ and a closed subspace $\St \subseteq \HH$ the Schur complement $\WS$ of $W$ to $\St,$ was introduced by M. G. Krein in \cite{Krein}. If $\leq$ denotes the order in $L(\HH)$ induced by $L(\HH)^+,$ he proved that the set
$\{ X \in L(\HH): \ 0\leq X\leq W \mbox{ and } \ran X\subseteq \St\}$ has a maximum element  denoted by $\WS.$ The notion was later rediscovered by Anderson and Trapp in \cite{Shorted2}, where they also proved that
\begin{equation} \label{shorted}
	\WS=\inf \ \{ E^*WE: E^2=E \in L(\HH), \ \ker E=\St^\bot\}.
\end{equation}

We  show that 	$\WS$ is also the infimum of a suitable set of closed multivalued projections with fixed nullspace $\St^{\perp}.$ 
\begin{lema} \label{propScomp3} Let $W\in L(\HH)^+.$  If $W$ is $\St$-complementable then $W^{1/2}\pas=W^{1/2}\cpas \in L(\HH)$ and  $W\pas\in L(\HH)^+.$ 
\end{lema}

\begin{proof} If $W$ is $\St$-complementable, $\pas^*=\cpas^*$ is a closed projection with closed domain and, by Proposition \ref{propScomp}, $\ran W \subseteq \dom \pas^*.$ Then $\clran W \subseteq \dom \pas^*.$  Hence $\ran W^{1/2} \subseteq \dom \pas^*$ so that $\dom \pas^* W^{1/2}=\HH.$ Using Lemma \ref{productolema}, it follows that
	$$(\pas^*W^{1/2})^*=W^{1/2} \cpas.$$ 
	Then $W^{1/2}\cpas$ is an everywhere defined closed operator so that $W^{1/2}\cpas \in L(\HH).$ 
	Since $W^{1/2}\pas \subseteq W^{1/2}\cpas$ and $\dom W^{1/2}\pas=\HH$ it follows that $W^{1/2}\pas= W^{1/2}\cpas \in L(\HH).$
	
	Finally, using Proposition \ref{propScomp} and Lemma \ref{productolema}, $$W\pas=W\pas^2=\pas^*W\pas=(W^{1/2}\pas)^* (W^{1/2}\pas)\in L(\HH)^+.$$
\end{proof}

\begin{prop} \label{propinfSchur} Let $W\in L(\HH)^+$ and $\St$ be a closed subspace of $\HH.$ Then
	$$\WS=\inf \ \{ Q^*WQ: Q \in \SP(\HH), \dom Q=\HH, \ \ker Q=\St^{\perp}, \ \mul Q \subseteq \ker W\}.$$	
\end{prop}

\begin{proof} Take $Q \in \SP(\HH)$ with $\dom Q=\HH,$ $\ker Q=\St^{\perp}$ and $\mul Q \subseteq \ker W.$ By Proposition \ref{PyI-P}, 
	$$Q=Q_0 \oplus Q_{\mul},$$ where $Q_0$ is a (closed) projection with $\ker Q_0=\ker Q=\St^{\perp}.$ But $\dom  Q_0=\HH$ so that $Q_0 \in L(\HH).$	Therefore, since $\mul Q \subseteq \ker W,$ $W^{1/2}Q=W^{1/2}Q_0,$ and then
	$$Q^*WQ=(W^{1/2}Q)^*(W^{1/2}Q)=Q_0^*WQ_0.$$ So, by \eqref{shorted}, the result follows.
\end{proof}

\begin{cor} \label{SchurWS} Let $W\in L(\HH)^+.$ If $W$ is $\St$-complementable then $$\WSb=W(I-\pas).$$
\end{cor}  
\begin{proof} Since $I-\cpas$ is closed and $\mul (I-\cpas) \subseteq \ker W,$ we get that $W(I-\cpas)=W(I-\cpas)_0,$ where $(I-\cpas)_0$ is the operator part of $I-\cpas,$ see Proposition \ref{PyI-P}. 
	Then, by \cite[Proposition 5.1]{CM}, 
	$$\WSb=W(I-\cpas)_0=W(I-\cpas)=W(I-\pas),$$ where we used Lemma \ref{propScomp3} for the last equality.
\end{proof}

\section{Weighted least squares solutions of inclusions}  \label{fifth}

In what follows, $W\in L(\HH)^+$ and $A\in \lr(\HH)$ are given. We consider the semi-norm
$$\Vert x \Vert_W:=\PI{x}{x}_W^{1/2}=\Vert W^{1/2} x\Vert, \ x \in \HH.$$
Recall that $Ax=\{y \in \HH: (x,y) \in A\}=y+\mul A$ for any $(x,y) \in A.$

\begin{Def}Given $b\in\HH$ a vector $x_0\in\HH$ is a $W$-\emph{least squares solution} ($W$-LSS) of the inclusion $b\in Ax$ if $x_0\in\dom A$ and there exists $z\in Ax_0$ such that
	\begin{equation}\label{ALSS}
		||z-b||_{W}= \underset{y \in \ran A}{\min} \ \ ||y-b||_{W}.
	\end{equation}
\end{Def}

\begin{obs} \label{obsLSS} A vector $x_0$ is a solution of the inclusion $b \in Ax$ if and only if $x_0$ is a solution of  the equation $Ax=b+\mul A.$ In fact, if $x_0$ is a solution of the inclusion  then $(x_0,b)\in A$ so that $Ax_0=b+\mul A.$ Conversely, if $Ax_0=b+\mul A$ then $b \in Ax_0,$ or equivalently $x_0$ is a solution of the inclusion $b \in Ax.$ 
\end{obs}

The $W$-LSS of the equation $Ax=b$ with $A\in L(\HH)$ were studied in \cite{CM}. On the other side, the  LSS of the inclusion $b\in Ax$ were studied in \cite{LN}. Here, we discuss the $W-$LSS of $b\in Ax$ by means of the multivalued projection $P_{W,\ran A}$.  

For an operator $A\in L(\HH)$ and $b\in\HH,$ there is a LSS of the equation $Ax=b$ if and only if $b \in \ran A \oplus \ran A^{\perp}$ \cite{Nashed}. The analogue for $W$-LSS of inclusions is the following. 

\begin{prop}\label{ALSS5} Let $b\in \HH$. There exists a $W$-LSS of $b\in Ax$ if and only if  $b\in\dom P_{W,\ran A}.$ Therefore, there exists a $W$-LSS of $b\in Ax$ for every $b \in \HH$ if and only if $W$ is $\ran A$-complementable.
\end{prop}

\begin{proof}
	Using the well-known fact that given $w_0\in\HH$ and  a subspace $\M\subseteq \HH$, $d(w_0, \M)=\inf_{m\in \M}||w_0-m||=||w_0-m_0||,$  for $m_0\in \M$ if and only if $w_0-m_0\in\M^\bot$, and the definition of $W$-LSS of an inclusion, the result follows. In fact, $x_0$ is a $W$-LSS of $b\in Ax$ if and only if there exists $z\in Ax_0$ such that $||W^{1/2}z-W^{1/2}b||=d(W^{1/2}b, W^{1/2}\ran A). $ But this is equivalent to $W^{1/2}(z-b)\in (W^{1/2}\ran A)^\bot$ or $z-b\in W^{-1}(\ran A^\bot)$, or $b\in \ran A+W^{-1}(\ran A^\bot)=\dom P_{W,\ran A}.$
	Finally, there exists a $W$-LSS for every $b \in \HH$ if and only if $\dom P_{W, \ran A}=\HH,$ or equivalently $W$ is $\ran A$-complementable.
\end{proof}

Recall that if $W \in L(\HH)^+$ and $\St$ is  a subspace of $\HH$ the relation $W^{1/2}(I-P_{W,\St})$ is in fact an operator,  see Lemma \ref{propPas}, 1. 
\begin{prop}\label{ALSS5II} Let $\St$ be a subspace of $\HH$ and let $b\in\dom P_{W,\St}.$ Then
	$$\min_{y\in\St}||y-b||_{W}=\Vert W^{1/2} (I-P_{W,\St}) b\Vert$$ and the minimum is attained at $y\in \St$ if and only if $y \in  P_{W,\St}b.$ 	
	
	If $W$ is $\St$-complementable then
	$$\min_{y\in\St}||y-b||_{W}=\Vert (\WSb)^{1/2}b \Vert.$$
\end{prop}	

\begin{proof}
	If $b\in\dom P_{W,\St},$ using that $W^{1/2}P_{W, \St}$ is an operator, by \cite[Proposition I.4.2 (e)]{Cross1}, we can write $W^{1/2}-W^{1/2}P_{W, \St}=W^{1/2}(I-P_{W, \St}).$ Then
	\begin{align*}
		W^{1/2}b=W^{1/2}P_{W, \St}b+(W^{1/2}-W^{1/2}P_{W, \St})b=W^{1/2}P_{W, \St}b+W^{1/2}(I-P_{W, \St})b.
	\end{align*}
	
	So that, for any $y\in \St,$
	\begin{align}\label{eq611}
		\Vert y -b \Vert_W^2&=\Vert W^{1/2}y-W^{1/2}b\Vert^2\nonumber\\
		&=\Vert W^{1/2}y-W^{1/2}P_{W, \St}b\Vert^2+\Vert W^{1/2}(I-P_{W, \St})b\Vert^2\nonumber\\
		&\geq \Vert W^{1/2}(I-P_{W, \St})b\Vert^2,
	\end{align}
	because $W^{1/2}(\St)$ and $W^{1/2}(\ker P_{W, \St})$ are orthogonal subspaces. 
	Moreover, equality is attained in \eqref{eq611} taking any $y \in P_{W, \St}b.$ In fact
	$$W^{1/2}P_{W, \St}b=W^{1/2}y \mbox{ for any } y\in P_{W, \St}b.$$ To see this, take $y\in P_{W, \St}b$ then $P_{W, \St}b=y+\mul P_{W, \St}.$ If $z \in P_{W, \St}b$ then $z=y+n,$ $n \in \ker W.$ So that $W^{1/2}z=W^{1/2}y.$
	
	Furthermore,  $y \in \St$ is such that $\Vert y- b\Vert_W=\min_{y\in\St}||y-b||_{W}$ if and only if $W^{1/2}y-W^{1/2}P_{W, \St}b=0$ if and only if $y-y'\in \ker W \cap \St=\mul P_{W, \St}$ for any $y' \in P_{W, \St}b,$  or equivalently, $y\in y'+\mul P_{W, \St}=P_{W, \St}b.$
	
	If $W$ is $\St$-complementable, by Corollary \ref{SchurWS}, $\WSb=W(I-P_{W, \St}).$ Then, by Lemma \ref{propScomp3}, 
	\begin{align*}
		\Vert W^{1/2}(I-P_{W, \St})b\Vert^2&=\PI{W^{1/2}(I-P_{W, \St})b}{W^{1/2}(I-P_{W, \St})b}\\
		&=\PI{(W^{1/2}(I-P_{W, \St}))^*W^{1/2}(I-P_{W, \St})b}{b}\\
		&=\Vert (\WSb)^{1/2}b \Vert^2
	\end{align*} and the result follows.
\end{proof}


\begin{prop}\label{ALSS3} Let $b\in\dom P_{W,\ran A}$. Then the following are equivalent:
	\begin{enumerate}
		\item $x_0$ is a $W-$LSS of the inclusion $b\in Ax$.
		\item $A x_0\cap P_{W, \ran A}b$ is not empty.
		\item $x_0\in A^{-1}P_{W, \ran A}b.$
		\item $x_0$ is a solution of the inclusion $0 \in A^*W(Ax-b).$
	\end{enumerate}
\end{prop}
\begin{proof} 
	
	Set $Q:= P_{W,\ran A}.$ By Proposition \ref{ALSS5}, $\underset{y \in \ran A}{\min} \ \ ||y-b||_{W}$ is attained at any   $y\in  Qb .$ Consequently, $x_0$ is a $W$-LSS of the inclusion $b\in Ax$ if and only if there exists $y\in A x_0\cap Qb.$ Thus, $i)$ and $ii)$ are equivalent. On the other hand, there exists $y\in Ax_0\cap Qb$ if and only if there exists $y\in\HH$ such that $(x_0,y)\in A$ and $(b,y)\in Q$ or equivalently, $(b,x_0)\in A^{-1}Q. $ This means that $ii)$ and $iii)$ are equivalent.
	
	$ii)\Leftrightarrow iv)$: If $Ax_0 \cap Qb$ is not empty then there exists $y$ such that $(x_0,y) \in A$ and  $(b,y) \in Q.$ Therefore $b=y+w$ for $w \in (W \ran A)^{\perp}.$
	Then $y-b=-w \in (Ax_0-b) \cap W^{-1}(\ran A^{\perp}),$ so that $W(y-b)\in W(Ax_0-b) \cap \ran A^{\perp}=W(Ax_0-b) \cap \ker A^*.$ Hence $(W(y-b),0) \in A^*$ then $0 \in A^*W(y-b)$ and so $0 \in A^*W(Ax_0-b).$
	Conversely, suppose that $0 \in A^*W(Ax_0-b).$ Then there exists $y \in Ax_0$ such that $(y-b,0) \in A^*W,$ or $y-b \in W^{-1}(\ran A^{\perp}).$ Then $y=b+w$ for $w \in W^{-1}(\ran A^{\perp}).$ But, since $b \in \dom Q,$ then $b=z+w',$ $z \in \ran A$ and $w' \in (W\ran A)^{\perp},$ so that $z \in Qb$ and $y=z+w+w'$ and, since $y \in \ran A,$ it follows that $w+w' \in \mul P_{W,\ran A}.$ Finally, since $Qb=z+ \mul Q,$ we get that $y \in Qb \cap Ax_0.$
\end{proof}

When $A \in L(\HH)$, it is well known that $x_0$ is a LSS of $Ax=b$ if and only if $Ax_0=P_{\clran A}b.$ Or equivalently, $x_0$ is a solution of the normal equation $A^*(Ax-b)=0.$ The set of LSS of $Ax=b$ is $x_0+\ker A$ where $x_0$ is a particular LSS \cite{GrevilleGI}.
Similar results were given in \cite{CM} when a positive weight is considered, with the additional hypothesis of $\St$-complementability. For inclusions we have the following result.
\begin{cor}\label{corALSS3} Let $b\in\dom P_{W,\ran A}$. Then the following are equivalent:
	\begin{enumerate}
		\item $x_0$ is a $W-$LSS of the inclusion $b\in Ax.$
		\item $Ax_0+\mul P_{W, \ran A}=P_{W, \ran A}b+\mul A.$
		\item $x_0$ is a solution of the normal equation $$A^*W(Ax-b)=A^*W(\mul A).$$
	\end{enumerate}
\end{cor}

\begin{proof} 	Set $Q:= P_{W,\ran A}.$ By Proposition \ref{ALSS3}, $i)$ holds if and only if $Ax_0 \cap Qb \not = \emptyset.$ 
	
	$i)\Leftrightarrow ii)$: If $i)$ holds, let $y \in Ax_0 \cap Qb.$ Then $Ax_0=y+\mul A$ and $Qb=y+\mul Q.$ Hence $Ax_0+\mul Q=Qb+\mul A.$ Conversely, if $Ax_0+\mul Q=Qb+\mul A$ then $Qb \subseteq Ax_0+\mul Q.$  Then, if $y \in Qb$  there exists $y' \in Ax_0$ and $z\in \mul Q$ such that $y=y'+z.$ Then $y-z=y' \in Qb+\mul Q=Qb$ so that $y-z=y'\in Qb \cap Ax_0$ and $i)$ holds. 
	
	$i)\Leftrightarrow iii)$: If $i)$ holds, let $y \in Ax_0 \cap Qb.$ Then $(x_0,y) \in A$ and $b=y+z$ for $z \in \ker Q.$ Since $Ax_0=y+\mul A,$ it follows that
	\begin{align*}
		A^*W(Ax_0-b)&=A^*W(y+\mul A-y-z)=A^*W(\mul A-z)\\
		&=A^*W(\mul A)+A^*Wz=A^*W(\mul A),
	\end{align*}
	because $A^*Wz \in A^*W(\ker Q)=A^*W(W^{-1}(\ran A^{\perp})) \subseteq A^*(\ran A^{\perp})=\{0\}.$ Conversely, if $A^*W(Ax_0-b)=A^*W(\mul A)$ then $0 \in A^*W(Ax_0-b)$ because $0 \in \mul A^*W \subseteq A^*W(\mul A).$ Then, by item $iv)$ of Proposition \ref{ALSS3}, $i)$ holds.
\end{proof}

\begin{prop}\label{ALSS2}Let $b\in\dom P_{W,\ran A}$ and $x_0$ be a $W$-LSS solution of $b\in Ax$. Then the set of $W$-LSS of the inclusion $b \in Ax$ is $x_0+A^{-1}(\mul P_{W,\ran A})=x_0+A^{-1}\ker W.$
\end{prop}
\begin{proof} Let $x_0$ be a $W$-LSS solution of $b\in Ax$. By Proposition \ref{ALSS3}, $x_1$ is a $W$-LSS solution of $b\in Ax$ if and only if $x_0-x_1\in \mul A^{-1}(P_{W,\ran A}).$ But, by Lemma \ref{propbasicas}, $\mul A^{-1}P_{W,\ran A}=A^{-1}\mul P_{W,\ran A}=A^{-1} (\ran A \cap \ker W)=A^{-1}\ker W.$ 
\end{proof}

The above result can be used to describe the $W_1W_2-$LSS of inclusions. See also \cite{CM}.

\begin{Def} Let $W_1, W_2 \in L(\HH)^+.$ Then  $x_0 \in \HH$ is a $W_1 W_2$-LSS of $b\in Ax$ if $x_0$ is a $W_1$-LSS of $b\in Ax$ and $||x_0||_{W_2}\leq ||\tilde{x}||_{W_2}$ for every $\tilde{x}$ which is a $W_1$-LSS of $b\in Ax.$
\end{Def}

\begin{prop}  Let $b\in \dom P_{W_1,\ran A}.$ Then $x_0\in \HH$ is a $W_1W_2-$LSS of  $b\in Ax$ if and only if  $x_0\in (I-P_{W_2, A^{-1}\ker W_1})A^{-1}P_{W_1,\ran A}b.$
\end{prop}
\begin{proof}
	Assume that $x_0\in \HH$ is a $W_1W_2$-LSS of the inclusion $b\in Ax$. Then $x_0$ is a $W_1$-LSS of $b \in  Ax$ and, by Proposition \ref{ALSS3}, $x_0\in  A^{-1}P_{W_1,\ran A}b.$ By Proposition \ref{ALSS2}, the set of  $W_1-$LSS is $x_0+A^{-1}\ker W_1.$ By the definition of $W_1W_2-$LSS, $x_0$ also solves $\underset{w\in A^{-1}\ker W_1}{\min}||x_0-w||_{W_2}.$ Since the minimum is attained at $x_0=x_0+0,$ by Proposition \ref{ALSS5II}, $0=P_{W_2, A^{-1}\ker W_1}x_0.$
	Therefore,  $x_0\in (I-P_{W_2, A^{-1}\ker W_1})x_0.$ So $x_0\in (I-P_{W_2, A^{-1}\ker W_1})A^{-1}P_{W_1,\ran A}b.$
	
	Conversely, let $x_0\in (I-P_{W_2, A^{-1}\ker W_1})A^{-1}P_{W_1,\ran A}b$. Then, consider $x_1\in A^{-1}P_{W_1,\ran A}b$ such that $x_0\in (I-P_{W_2, A^{-1}\ker W_1})x_1$. Thus, $x_0-x_1\in P_{W_2, A^{-1}\ker W_1}x_1$ and, by Proposition \ref{ALSS5II},  $||x_0||_{W_2}\leq ||x_1+y||_{W_2}$ for all $y\in A^{-1}\ker W_1$. Therefore, by Propositions \ref{ALSS3} and \ref{ALSS2}, $x_0$ is a $W_1W_2-$LSS of  $b\in Ax$.
\end{proof}
\subsection{Abstract splines and smoothing problems}
In this section we apply our previous study on $W$-LSS of inclusions to solve a classical interpolating problem introduced by Atteia in \cite{Atteia} and associated regularized version.

Let $T\in L(\HH, \E)$ with closed range and $V\in L(\HH, \KK)$ surjective. Given $b\in\HH,$ consider the problem:
\begin{equation}\label{l}
	\min ||Tx||, \; \ \ \text{subject to} \;  Vx=b.
\end{equation}
A$(T, V)$-abstract spline to $b$ is
any element of the set 
$$\spl(T,V,b)=\left\{y\in \HH: Vy=b,  ||Ty||=\underset{Vx=b}{\min} \ ||Tx||\right\}.$$

Since $V$ is surjective, given $\tilde{x}\in\HH$ such that $V\tilde{x}=b$, we can  write the above problem as:
\begin{center}
	$\underset{w\in\ker V}{\min} ||T(\tilde{x}-w)||.$ 
	
\end{center}
Equivalently, $\underset{Vx=b}{\min} \ ||Tx||^2=\underset{w\in\HH}{\min} \ ||T(\tilde{x}-P_{\ker V}w)||^2=\underset{w\in\HH}{\min} \ ||\tilde{x}-P_{\ker V}w||_{T^*T}^2.$ Thus, the set $\spl(T,V,b)$ is related to $T^*T-$LSS in the following fashion.

\begin{lema}\label{ALSSsplines} Let $\tilde{x}\in\HH$ be such that $V\tilde{x}=b$. Then, $y\in \spl(T,V,b)$ if and only if $y=\tilde{x}-P_{\ker V}w$ where $w$ is a $T^*T-$LSS of $\tilde{x}=P_{\ker V}x$. 
	
\end{lema}

In \cite{CMS}, the existence of splines for all $\xi\in\HH$ is characterized in terms of the existence of $T^*T$-self adjoint projections when $T^*T$ is $\ker V$-complementable. Here, we propose a simpler characterization of $\spl(T,V,b)$ by means of multivalued projections.  Compare the following result with \cite[Theorem 3.2]{CMS}.
\begin{prop}\label{ALSS1} \label{2} Let $\tilde{x}\in\HH$ be such that $V\tilde{x}=b$. Then, $\spl(T, V, b)$ is not empty if and only if $\tilde{x}\in\dom P_{T^*T,\ker V}.$ If $\tilde{x}\in \dom P_{T^*T,\ker V}$ then
	\begin{equation}
		\spl(T,V,b)=(I- P_{T^*T,\ker V})\tilde{x}.
	\end{equation}
\end{prop}

\begin{proof}

	The proof of the first part follows by Lemma \ref{ALSSsplines} and Propositions \ref{ALSS5} and	\ref{ALSS3}.	Finally,  $I-P_{T^*T,\ker V}$ is decomposable because $\ker V$ and $(T^*T\ker V)^\bot$ are closed subspaces. Moreover, by Proposition \ref{PyI-P},   $I-P_{T^*T,\ker V}=P_0\hat{\oplus}(\{0\}\times \mul P_{T^*T,\ker V})$ where $P_0:=P_{(T^*T\ker V)^\bot\ominus \mul P_{T^*T,\ker V} \pl \\ \ker V}.$ 
\end{proof}

The \emph{regularized problem} associated to (\ref{l}), known as the \emph{smoothing problem} (or Tikhonov regularization) \cite{Tikhonov} is
\begin{equation}\label{t}
	\underset{x\in\HH}{\min}\left(||Tx||^2+\rho ||Vx-b||^2\right)^{1/2},
\end{equation}
where $\rho>0$ is a  parameter.

We follow the notation used in \cite{Hassi} and given $A \in L(\mathcal E, \HH)$ and $B \in L(\mathcal E, \KK),$ we write
$$L(A,B):=\{(Ax,Bx) : x \in \mathcal E\}.$$

Thus, $L(A,B)=\ran \begin{pmatrix} A \\ B
\end{pmatrix}$ is a range operator and $L(A,B)=BA^{-1}$. That is, in the sense of the product of linear relations, $L(A,B)$ is a \emph{quotient}.

For each fixed $\rho,$ define the inner product in $\E\times\HH$
$$\left\langle (y,z), (y',z')\right\rangle_\rho:=\left\langle y,y'\right\rangle_{\E}+\rho \left\langle z,z'\right\rangle_{\HH},$$
$y,y'\in\E, z,z'\in\HH,$ and consider the associated norm $\Vert (y,z) \Vert_\rho.$ 
Then
\begin{eqnarray}
	\underset{x\in\HH}{\min}\left(||Tx||^2+\rho ||Vx-b||^2\right)^{1/2}&=&\underset{x\in\HH}{\min}||(Tx,Vx)-(0,b)||_{\rho}\nonumber\\
	&=&d_\rho(VT^{-1}, (0,b)).\nonumber
\end{eqnarray}

If  $VT^{-1}$ is closed, then
$$d_\rho(VT^{-1}, (0,b))=||(I-P_{VT^{-1}})(0,b)||_{\rho}.$$

But $VT^{-1}=\ran \begin{pmatrix} T \\ V\end{pmatrix} $ is closed if and only if $\ran \begin{pmatrix} T^* & V^*\end{pmatrix}=\ran T^*+\ran V^*$ is closed if and only if  $\ker T+\ker V$ is closed. Or equivalently,  $\begin{pmatrix}
	T^* & V^*
\end{pmatrix} \begin{pmatrix}
	T\\ V
\end{pmatrix}=T^*T+V^*V$ has closed range \cite{Hassi}.

Following the ideas in the proof of \cite[Lemma 3.6]{Hassi}, consider $T':=T((T^*T+V^*V)^{1/2})^\dagger$ and $V':=V((T^*T+V^*V)^{1/2})^\dagger$, where $A^\dagger$ denotes the Moore-Penrose inverse of the operator $A$ and $A^\dagger$ is bounded if and only if $A$ has closed range. Then,  $VT^{-1}=V'T'^{-1}$ and  $T'^*T'+V'^*V'=P_{\ran(T^*T+V^*V)}$ and, by \cite[Lemma 3.4]{Hassi}, the orthogonal projection $P_{VT^{-1}}$ is given by
\begin{eqnarray}\label{m}
	P_{VT^{-1}}&=&\begin{pmatrix} T'T'^* & T'V'^* \\ V'T'^* & V'V'^*\end{pmatrix}\nonumber\\
	&=&\begin{pmatrix} T(T^*T+V^*V)^\dagger T^* & T(T^*T+V^*V)^\dagger V^* \\V(T^*T+V^*V)^\dagger T^* & V(T^*T+V^*V)^\dagger V^*\end{pmatrix}.
\end{eqnarray}
This proves the following.
\begin{prop}
	Let $T\in L(\HH,\E)$ with closed range and $V\in L(\HH,\KK)$ surjective. If $\ker V+\ker T$ is closed then the minimum in (\ref{t}) is given by
	$$||(I-P_{VT^{-1}})(0,b)||_{\rho},$$
	where $P_{VT^{-1}}$ is given by (\ref{m}).
\end{prop}

\subsection*{Acknowledgements}
M.~L.~Arias was supported in part by FonCyT (PICT 2017-0883) and UBACyT (20020190100330BA). M.~Contino was supported by María Zambrano Postdoctoral Grant CT33/21 at Universidad Complutense de Madrid financed by the Ministry of Universities with Next Generation EU funds.
A. ~Maestripieri was supported in part by the Interdisciplinary Center for Applied Mathematics at Virginia Tech.
M.~Contino, A.~Maestripieri  and S.~Marcantognini were
supported by CONICET PIP 11220200102127CO.  

\bibliographystyle{amsplain}

\end{document}